\documentclass{amsart}
\usepackage{amssymb}

\usepackage{amscd}
\usepackage{thmdefs}
\usepackage[dvips]{epsfig}

\input{tcilatex}
\theoremstyle{definition}
\theoremstyle{remark}
\numberwithin{equation}{section}

\input tcilatex

\begin{document}
\title[Matricial Representations of $\Gamma _{N}^{2}$]{Matricial Representations of Certain Finitely Presented Groups Generated by
Order-2 Generators and Their Applications}
\author{Ryan Golden and Ilwoo Cho}
\address{Saint Ambrose Univ., Dept. of Math. \& Stat., 518 W. Locust St., Davenport,
Iowa, 52803, U. S. A. / }
\email{goldenryanm@sau.edu / choilwoo@sau.edu}
\date{}
\dedicatory{}
\thanks{}
\maketitle

\begin{abstract}
In this paper, we study matricial representations of certain finitely
presented groups $\Gamma _{N}^{2}$ with $N$-generators of order-2. As an
application, we consider a group algebra $\mathcal{A}_{2}$ of $\Gamma
_{2}^{2},$ under our representations. Specifically, we characterize the
inverses $g^{-1}$ of all group elements $g$ in $\Gamma _{2}^{2},$ in terms
of matrices in the group algebra $\mathcal{A}_{2}$. From the study of this
characterization, we realize there are close relations between the trace of
the radial operator of $\mathcal{A}_{2},$ and the Lucas numbers appearing in
the Lucas triangle.
\end{abstract}

\strut

\section{Introduction}

In this paper, we consider certain\emph{\ finite-dimensional representations 
}$\left( \Bbb{C}^{n},\text{ }\alpha ^{n}\right) $ of a\emph{\ finitely
presented group} $\Gamma _{N}^{2},$

(1.0.1)

\begin{center}
$\Gamma _{N}^{2}$ $\overset{def}{=}$ $\left\langle \{x_{1},\text{ ..., }%
x_{N}\},\text{ }\{x_{1}^{2}=...=x_{N}^{2}\}\right\rangle ,$
\end{center}

\strut where the generators $x_{1},$ ..., $x_{N}$ are \emph{noncommutative
indeterminants}. \strut i.e., an algebraic structure $\Gamma _{N}^{2}$ is the%
\emph{\ group} generated by $N$-\emph{generators} $x_{1},$ ..., $x_{N}$
equipped with its ``noncommutative'' binary operation ($\cdot $), satisfying

\begin{center}
$x_{1}^{2}$ $=$ $x_{2}^{2}$ $=$ ... $=$ $x_{N}^{2}$ $=$ $e_{N},$
\end{center}

where $e_{N}$ is the \emph{group-identity of} $\Gamma _{N}^{2}$ such that:

\begin{center}
$ge_{N}$ $=$ $g$ $=$ $e_{N}g,$ for all $g$ $\in $ $\Gamma _{N}^{2},$
\end{center}

\strut for all $N$ $\in $ $\Bbb{N}.$ So, one may understand a group $\Gamma
_{N}^{2}$ as the \emph{quotient group},

(1.0.2)

\begin{center}
$\Gamma _{N}^{2}$ $\overset{\text{Group}}{=}$ $\mathcal{F}_{N}$ $/$ $%
\{g_{j}^{2}=e:g_{j}$ are generators of $\mathcal{F}_{N}\},$
\end{center}

where $\mathcal{F}_{N}$ is the\emph{\ noncommutative }(\emph{non-reduced})%
\emph{\ free group }$\left\langle \{g_{j}\}_{j=1}^{N}\right\rangle $\emph{\
with }$N$\emph{-generators} $g_{1},$ ..., $g_{N}$, where $e$ is the
group-identity of the free group $\mathcal{F}_{N},$ and ``$\overset{\text{%
Group}}{=}$'' means ``being group-isomorphic'' (e.g., [1] and [2]).

By construction, all elements $g$ of the group $\Gamma _{N}^{2}$ are of
the form,

(1.0.3)

\begin{center}
$g$ $=$ $x_{i_{1}}^{k_{1}}$ $x_{i_{2}}^{k_{2}}$ ... $x_{i_{n}}^{k_{n}},$
\end{center}

for some

\begin{center}
$i_{1},$ ..., $i_{n}$ $\in $ $\{1,$ ..., $N\},$ and $k_{1},$ ..., $k_{n}$ $%
\in $ $\Bbb{Z},$
\end{center}

as non-reduced words in $\{x_{1},$ ..., $x_{N}\},$ where $x_{j}^{-1}$ means
the \emph{group-inverse of} $x_{j},$ and hence, $x_{j}^{-k}$ means $\left(
x_{j}^{-1}\right) ^{k},$ for all $k$ $\in $ $\Bbb{N},$ for all $j$ $=$ $1,$
..., $N.$

\strut In this paper, for an arbitrarily fixed $N$ $\in $ $\Bbb{N},$ we
establish and study certain $n$-dimensional Hilbert-space representations $%
\left( \Bbb{C}^{n},\text{ }\alpha ^{n}\right) $ of the groups $\Gamma
_{N}^{2},$ for all $n$ $\in $ $\Bbb{N}$ $\setminus $ $\{1\}.$ Under our
representations, each element $g$ of $\Gamma _{N}^{2}$ is understood as a
matrix $A_{g}$ acting on $\Bbb{C}^{n},$ for all $n$ $\in $ $\Bbb{N}.$
Moreover, if

\begin{center}
$g$ $=$ $x_{i_{1}}^{k_{1}}...x_{i_{l}}^{k_{l}}$ as in (1.0.3)
\end{center}

(as a ``non-reduced'' word in $\Gamma _{N}^{2}$), then there exists
corresponding matrices $A_{i_{1}},$ ..., $A_{i_{l}},$ such that

\begin{center}
$A_{g}$ $=$ $A_{i_{1}}^{k_{1}}...A_{i_{l}}^{k_{l}}$ in $M_{n}(\Bbb{C})$
\end{center}

(as a ``reduced'' word in $M_{n}(\Bbb{C})$), where $M_{n}(\Bbb{C})$ is the 
\emph{matricial algebra} consisting of all $(n\times n)$-matrices, for all $%
n $ $\in $ $\Bbb{N}.$

\subsection{\strut Motivation}

\strut The readers may skip reading this sub-section. Here, we simply want
to emphasize why we are interested in the groups $\Gamma _{N}^{2},$ for $N$ $%
\in $ $\Bbb{N}.$

\emph{Free probability} is a branch of \emph{operator algebra}. By
considering free-distributional data, one can establish operator-valued 
\emph{noncommutative} probability theory on (topological, or pure-algebraic) 
\emph{algebras} as in classical probability theory (e.g., [3], [4] and cited
papers therein). Also, such free-probabilistic data let us have structure
theorems of given algebras under \emph{free product} determined by given 
\emph{linear functionals}. Here, the \emph{independence} of classical
probability is replaced by the so-called \emph{freeness}. There are two
approaches in free probability theory: \emph{Voiculescu}'s original analytic
approach (e.g., [4]), and \emph{Speicher}'s combinatorial approach (e.g.,
[3]).

The second-named author, Cho, has considered connections between operator
algebra theory, and Hecke-algebraic number theory via free probability
recently, to provide tools for studying number-theoretic results with
operator-algebraic techniques, and vice versa, by establishing certain
representational and operator-theoretic models from combinatorial free
probability settings of [3].

Especially, in the series of research papers [5], [2] and [1], Cho
considered the \emph{free-probabilistic representations} of the \emph{Hecke
algebras }$\mathcal{H}(G_{p})$ generated by the \emph{generalized linear} ($%
(2\times 2)$-\emph{matricial}) \emph{groups} $G_{p}$ $=$ $GL_{2}(\Bbb{Q}%
_{p}) $ over the $p$-\emph{adic number fields} $\Bbb{Q}_{p},$ for \emph{%
primes} $p. $

In the frontier paper [5], Cho and Gillespie established free probability
models of certain subalgebras $\mathcal{H}_{Y_{p}}$ of the Hecke algebras $%
\mathcal{H}(G_{p})$ by defining suitable linear functionals on $\mathcal{H}%
(G_{p}).$ In particular, they constructed free-probabilistic structures
preserving number-theoretic data from $\mathcal{H}_{Y_{p}}.$

In [2], Cho extended free-probabilistic models of $\mathcal{H}_{Y_{p}}$ in
the sense of [6] to those of $\mathcal{H}(G_{p})$ fully, for all primes $p.$
On such models, the $C^{*}$-\emph{algebras} $\Bbb{H}(G_{p})$ were
constructed by realizing elements of $\mathcal{H}(G_{p})$ as operators under
free-probabilistic representations of $\mathcal{H}(G_{p}).$ And the
operator-theoretic properties of \emph{generating operators of} $\Bbb{H}%
(G_{p})$ were studied there.

In [1], by studying certain types of \emph{partial isometries} of $\Bbb{H}%
(G_{p}),$ Cho obtained the embedded non-abelian multiplicative groups $G_{N}$
of $\Bbb{H}(G_{p})$ generated by $N$-many partial isometries in the sense of
[1]. In particular, it was shown there that:

(1.1.1)

\begin{center}
$G_{N}$ $\overset{\text{Group}}{=}$ $\Gamma _{N}^{2},$ for all $N$ $\in $ $%
\Bbb{N},$
\end{center}

\strut where $\Gamma _{N}^{2}$ is in the sense of (1.0.1), satisfying
(1.0.2).

To study detailed algebraic properties of $G_{N}$ in $\mathcal{H}(G_{p}),$
and to investigate operator-algebraic properties of $C^{*}(G_{N})$ in $\Bbb{H%
}(G_{p}),$ he used the isomorphic group $\Gamma _{N}^{2}$ of (1.0.1), and
the corresponding group $C^{*}$-algebra $C^{*}\left( \Gamma _{N}^{2}\right)
. $

Thus, for authors, it is natural to be interested in the groups $\Gamma
_{N}^{2}$. In this paper, we study the groups $\Gamma _{N}^{2}$ of (1.0.1)
pure-algebraically (independent from those in [1], [2] and [5]).

\subsection{Overview}

In this paper, we concentrate on studying the groups $\Gamma _{N}^{2}$ of
(1.0.1) independently. The main purpose is to establish suitable \emph{%
Hilbert-space representations} (other than those of [1]) of $\Gamma
_{N}^{2}. $ Fundamentally, we construct ``finite-dimensional'' Hilbert-space
representations $\left( \Bbb{C}^{n},\text{ }\alpha ^{n}\right) $ of $\Gamma
_{N}^{2},$ for ``all'' $n$ $\in $ $\Bbb{N}$ $\setminus $ $\{1\}.$

Under our representation $\left( \Bbb{C}^{n},\text{ }\alpha ^{n}\right) $
for $n$ $\in $ $\Bbb{N},$ each group element $g$ $\in $ $\Gamma _{N}^{2}$ is
understood as a matrix $A_{g}$ in the matricial algebra $M_{n}(\Bbb{C}),$
consisting of all $(n\times n)$-matirices. To study an algebraic object $g$ $%
\in $ $\Gamma _{N}^{2},$ \strut we will investigate functional properties of
the corresponding matrices $A_{g}$ $\in $ $M_{n}(\Bbb{C})$ for $g,$ for all $%
n$ $\in $ $\Bbb{N}$ (See Section 2).

As application, we consider a group algebra $\mathcal{A}_{2}$ of $\Gamma
_{2}^{2},$ under our representations. Specifically, we characterize the
inverses $g^{-1}$ of all group elements $g$ in $\Gamma _{2}^{2},$ in terms
of matrices in the group algebra $\mathcal{A}_{2}$ generated by $\Gamma
_{2}^{2}$ (See Section 3).

From the study of this characterization of Section 3, we show that there are
close relations between the\emph{\ trace} of the powers $T^{n}$ of the \emph{%
radial operator} $T$ of $\mathcal{A}_{2},$ and the \emph{Lucas numbers} in
the \emph{Lucas triangle} (See Section 4).

\section{\strut Finite-Dimensional Representations of $\Gamma _{N}^{2}$}

In this section, we establish finite-dimensional Hilbert-space
representations of the finitely presented groups $\Gamma _{N}^{2}$ with $N$%
-generators of order-2 in the sense of (1.0.1). i.e.,

\begin{center}
$\Gamma _{N}^{2}$ $=$ $\left\langle \{x_{j}\}_{j=1}^{N},\text{ }%
\{x_{j}^{2}=e_{N}\}_{j=1}^{N}\right\rangle ,$
\end{center}

where $x_{1},$ ..., $x_{N}$ are noncommutative indeterminants (as
generators), and $e_{N}$ are the group-identities of $\Gamma _{N}^{2},$ for
all $N$ $\in $ $\Bbb{N}.$

Fix $N$ $\in $ $\Bbb{N}$ throughout this section and the corresponding group 
$\Gamma _{N}^{2}.$

\subsection{A Matrix Group $\frak{M}_{2}^{2}(N)$}

\strut Let $M_{2}(\Bbb{C})$ be the $(2\times 2)$-matricial algebra, and let

\begin{center}
$\Bbb{C}^{\times }$ $=$ $\Bbb{C}$ $\setminus $ $\{0\}$ in $\Bbb{C}.$
\end{center}

Suppose $\Bbb{C}$ $\times $ $\Bbb{C}^{\times }$ is the \emph{Cartesian
product} of $\Bbb{C}$ and $\Bbb{C}^{\times }.$ Define now a function

\begin{center}
$A_{2}$ $:$ $\Bbb{C}\times \Bbb{C}^{\times }$ $\rightarrow $ $M_{2}(\Bbb{C})$
\end{center}

by

(2.1.1)

\begin{center}
$A_{2}(a,$ $b)$ $=$ $\left( 
\begin{array}{cc}
a & b \\ 
\frac{1-a^{2}}{b} & -a
\end{array}
\right) ,$
\end{center}

for all $(a,$ $b)$ $\in $ $\Bbb{C}\times \Bbb{C}^{\times }.$ It is not
difficult to check that:

\begin{lemma}
Let $A$ be the map in the sense of (2.1.1). Then

(2.1.2)

\begin{center}
$\left( A_{2}(a,b)\right) ^{2}$ $=$ $I_{2},$ the identity matrix of $M_{2}(%
\Bbb{C}),$
\end{center}

for all $(a,$ $b)$ $\in $ $\Bbb{C}\times \Bbb{C}^{\times }.$
\end{lemma}

\begin{proof}
The proof of (2.1.2) is from the straightforward computation.
\end{proof}

\strut It is also easy to verify that$:$

\begin{center}
$A_{2}\left( a_{1},b_{1}\right) $ $A_{2}(a_{2},$ $b_{2})$ $=$ $\allowbreak
\left( 
\begin{array}{cc}
\frac{a_{1}a_{2}b_{2}+b_{1}-b_{1}a_{2}^{2}}{b_{2}} & a_{1}b_{2}-b_{1}a_{2}
\\ 
\frac{a_{2}b_{2}-a_{2}b_{2}a_{1}^{2}-a_{1}b_{1}+a_{1}b_{1}a_{2}^{2}}{%
b_{1}b_{2}} & \frac{b_{2}-b_{2}a_{1}^{2}+a_{1}a_{2}b_{1}}{b_{1}}
\end{array}
\right) ,$
\end{center}

and

\begin{center}
$A_{2}\left( a_{2},b_{2}\right) $ $A_{2}(a_{1},$ $b_{1})$ $=$ $\allowbreak
\left( 
\begin{array}{cc}
\frac{b_{2}-b_{2}a_{1}^{2}+a_{1}a_{2}b_{1}}{b_{1}} & b_{1}a_{2}-a_{1}b_{2}
\\ 
\frac{-a_{2}b_{2}+a_{2}b_{2}a_{1}^{2}+a_{1}b_{1}-a_{1}b_{1}a_{2}^{2}}{%
b_{1}b_{2}} & \frac{a_{1}a_{2}b_{2}+b_{1}-b_{1}a_{2}^{2}}{b_{2}}
\end{array}
\right) ,$
\end{center}

and hence,

(2.1.3)

\begin{center}
$A_{2}\left( a_{1},b_{1}\right) $ $A_{2}(a_{2},b_{2})$ $\neq $ $%
A_{2}(a_{2},b_{2})$ $A_{2}(a_{1},b_{1}),$
\end{center}

in general, for $(a_{k},$ $b_{k})$ $\in $ $\Bbb{C}\times \Bbb{C}^{\times },$
for $k$ $=$ $1,$ $2.$ In particular, whenever

\begin{center}
$(a_{k}$, $b_{k})$ $\in $ $\Bbb{C}^{\times }\times \Bbb{C}^{\times }$ in $%
\Bbb{C}\times \Bbb{C}^{\times },$
\end{center}

for $k$ $=$ $1,$ $2,$ and the pairs $(a_{1},$ $b_{1})$ and $(a_{2},$ $b_{2})$
are ``distinct'' in $\Bbb{C}^{\times }\times \Bbb{C}^{\times },$ the above
noncommutativity (2.1.3) always holds.

Therefore, the images

\begin{center}
$\{A_{2}(a,$ $b)$ $:$ $(a,$ $b)$ $\in $ $\Bbb{C}^{\times }\times \Bbb{C}%
^{\times }\}$
\end{center}

forms the noncommutative family in $M_{2}(\Bbb{C}),$ by (2.1.3).

Now, take distinct pairs,

(2.1.3)$^{\prime }$

\begin{center}
$(a_{k},$ $b_{k})$ $\in $ $\Bbb{C}^{\times }$ $\times $ $\Bbb{C}^{\times },$
for $k$ $=$ $1,$ ..., $N.$
\end{center}

For the above chosen $N$-pairs $\{(a_{k},$ $b_{k})\}_{k=1}^{N}$ of (2.1.3)$%
^{\prime },$ one can construct the corresponding matrices,

(2.1.4)

\begin{center}
$A_{2}(a_{k},$ $b_{k})$ in $M_{2}(\Bbb{C}),$ for $k$ $=$ $1,$ ..., $N,$
\end{center}

by (2.1.1). Under the matrix multiplication on $M_{2}(\Bbb{C}),$ let's
construct the multiplicative subgroup $\frak{M}_{2}^{2}(N)$ of $M_{2}(\Bbb{C}%
)$ as

(2.1.5)

\begin{center}
$\frak{M}_{2}^{2}(N)$ $=$ $\left\langle \{A_{2}(a_{k},\text{ }%
b_{k})\}_{k=1}^{N}\right\rangle $
\end{center}

generated by $A_{2}(a_{k},$ $b_{k}),$ for all $k$ $=$ $1,$ ..., $N.$ i.e.,
all elements of $\frak{M}_{2}^{2}(N)$ are the $(2\times 2)$-matrices in $%
M_{2}(\Bbb{C}),$ satisfying

(2.1.6)

\begin{center}
$\left( A_{2}(a_{k},b_{k})\right) ^{2}$ $=$ $I_{2}$ in $\frak{M}_{2}^{2}(N),$
\end{center}

for all $k$ $=$ $1,$ ..., $N,$ be (2.1.2).

\begin{definition}
\strut Let $\frak{M}_{2}^{2}(N)$ be the multiplicative subgroup (2.1.5) of $%
M_{2}(\Bbb{C}).$ We call $\frak{M}_{2}^{2}(N),$ the $2$-dimensional,
order-2, $N$-generator (sub-)group (of $M_{2}(\Bbb{C})$).
\end{definition}

\strut One obtains the following algebraic characterization.

\begin{theorem}
Let $\frak{M}_{2}^{2}(N)$ be a 2-dimensional, order-2, $N$-generator group
in $M_{2}(\Bbb{C}).$ Then the groups $\frak{M}_{2}^{2}(N)$ and the group $%
\Gamma _{N}^{2}$ of (1.0.1) are isomorphic. i.e.,

(2.1.7)

\begin{center}
$\frak{M}_{2}^{2}(N)$ $\overset{\text{Group}}{=}$ $\Gamma _{N}^{2}.$
\end{center}
\end{theorem}

\begin{proof}
Let $\frak{M}_{2}^{2}(N)$ be a 2-dimensional, order-2, $N$-generator group
in $M_{2}(\Bbb{C}).$ Then, by (2.1.6), each generator $A_{2}(a_{k},$ $b_{k})$
satisfies

\begin{center}
$\left( A_{2}(a_{k},\text{ }b_{k})\right) ^{2}$ $=$ $I_{2},$ for all $k$ $=$ 
$1,$ ..., $N.$
\end{center}

i.e., each generator $A_{2}(a_{k},$ $b_{k})$ has order-2 indeed in $\frak{M}%
_{2}^{2}(N).$

Since $(a_{k},$ $b_{k})$ are taken from $\Bbb{C}^{\times }\times \Bbb{C}%
^{\times },$ the generators $\{A_{2}(a_{k},$ $b_{k})\}_{k=1}^{N}$ forms a
noncommutative family in $M_{2}(\Bbb{C}),$ by (2.1.3).

Observe now that there does not exist $n$-tuple $(j_{1},$ ..., $j_{n})$ of
``distinct'' elements $j_{1},$ ..., $j_{n}$ in $\{1,$ ..., $N\},$ for all $n$
$\in $ $\Bbb{N},$ such that

\begin{center}
$\overset{n}{\underset{l=1}{\Pi }}$ $A_{2}\left( a_{k_{j_{l}}},\text{ }%
b_{k_{j_{l}}}\right) $ $=$ $A_{2}\left( a_{k_{j_{0}}},\text{ }%
b_{k_{j_{0}}}\right) ,$ or $I_{2},$
\end{center}

for some $j_{0}$ $\in $ $\{1,$ ..., $N\},$ by the very definition (2.1.1)
and (2.1.3).

These observations show that

$\frak{M}_{2}^{2}(N)$ $\overset{\text{Group}}{=}$ $\mathcal{F}_{N}$ $/$ $%
\mathcal{R}_{N},$

where $\mathcal{F}_{N}$ is the noncommutative free group $\left\langle
\{g_{j}\}_{j=1}^{N}\right\rangle $ with $N$-generators $g_{1},$ ..., $g_{N},$
and $\mathcal{R}_{N}$ means the relator set,

\begin{center}
$\mathcal{R}_{N}$ $=$ $\{g_{j}^{2}$ $=$ $e\}_{j=1}^{N},$
\end{center}

where $e$ is the group-identity of $\mathcal{F}_{N}.$

\strut

Recall that the group $\Gamma _{N}^{2}$ of (1.0.1) is group-isomorphic to
the quotient group $\mathcal{F}_{N}$ $/$ $\mathcal{R}_{N},$ by (1.0.2).
Therefore, one has that

\begin{center}
$\frak{M}_{2}^{2}(N)$ $\overset{\text{Group}}{=}$ $\mathcal{F}_{N}$ $/$ $%
\mathcal{R}_{N}$ $\overset{\text{Group}}{=}$ $\Gamma _{N}^{2}.$
\end{center}

Therefore, the 2-dimensional, order-2, $N$-generator group $\frak{M}%
_{2}^{2}(N)$ is group-isomorphic to the group $\Gamma _{N}^{2}$ of (1.0.1).
\end{proof}

\strut By (2.1.7), two groups $\frak{M}_{2}^{2}(N)$ and $\Gamma _{N}^{2}$
are isomorphic from each other, equivalently, there does exist a
group-isomorphism,

\begin{center}
$\alpha ^{2}$ $:$ $\Gamma _{N}^{2}$ $\rightarrow $ $\frak{M}_{2}^{2}(N),$
\end{center}

satisfying

(2.1.8)

\begin{center}
$\alpha ^{2}\left( x_{j}\right) $ $=$ $A_{2}\left( a_{j},\text{ }%
b_{j}\right) ,$ for all $j$ $=$ $1,$ ..., $N.$
\end{center}

The above structure theorem (2.1.7), equivalently, the existence of the
group-isomorphism $\alpha ^{2}$ of (2.1.8) provides a 2-dimensional
Hilbert-space representation $\left( \Bbb{C}^{2},\text{ }\alpha ^{2}\right) $
of $\Gamma _{N}^{2}.$

\begin{theorem}
There exists the $2$-dimensional Hilbert-space representation $\left( \Bbb{C}%
^{2},\text{ }\alpha ^{2}\right) $ of $\Gamma _{N}^{2}.$ Especially, $\alpha
^{2}\left( g\right) $ acting on $\Bbb{C}^{2}$ are in the sense of (2.1.8),
for all $g$ $\in $ $\Gamma _{N}^{2}$. i.e.,

(2.1.9) $\qquad \left( \Bbb{C}^{2},\text{ }\alpha ^{2}\right) $ forms a
2-dimensional representation of $\Gamma _{N}^{2},$

where $\alpha ^{2}$ is in the sense of (2.1.8).
\end{theorem}

\begin{proof}
Since $\alpha ^{2}$ of (2.1.8) is a generator-preserving group-isomorphism
from $\Gamma _{N}^{2}$ onto $\frak{M}_{2}^{2}(N),$ it satisfies

\begin{center}
$\alpha ^{2}\left( g_{1}g_{2}\right) $ $=$ $\alpha ^{2}\left( g_{1}\right) $ 
$\alpha ^{2}\left( g_{2}\right) $ in $\frak{M}_{2}^{2}(N),$
\end{center}

for all $g_{1},$ $g_{2}$ $\in $ $\Gamma _{N}^{2},$ and

\begin{center}
$\alpha ^{2}\left( g^{-1}\right) $ $=$ $\left( \alpha ^{2}\left( g\right)
\right) ^{-1}$ in $\frak{M}_{2}^{2}(N),$
\end{center}

for all $g$ $\in $ $\Gamma _{N}^{2}.$

Since $\frak{M}_{2}^{2}(N)$ $\subseteq $ $M_{2}(\Bbb{C}),$ the matrices $%
\alpha ^{2}\left( g\right) $ are acting on $\Bbb{C}^{2},$ for all $g$ $\in $ 
$\Gamma _{N}^{2}.$ Therefore, the pair $\left( \Bbb{C}^{2},\text{ }\alpha
^{2}\right) $ forms a Hilbert-space representation of $\Gamma _{N}^{2}.$
\end{proof}

\subsection{\strut Certain Order-2 Matrices in $M_{n}(\Bbb{C})$}

Now, let $n$ $>$ $2$ in $\Bbb{N}.$ And consider certain types of matrices in 
$M_{n}(\Bbb{C}).$ In Section 2.1, we showed that our group $\Gamma _{N}^{2}$
of (1.0.1) is group-isomorphic to the multiplicative subgroup $\frak{M}%
_{2}^{2}(N)$ of (2.1.5) in $M_{2}(\Bbb{C}),$ by (2.1.7), and hence, we
obtain a natural 2-dimensional representation $\left( \Bbb{C}^{2},\text{ }%
\alpha ^{2}\right) $ of $\Gamma _{N}^{2}$ in (2.1.9). In this section, we
construct a base-stone to extend the representation $\left( \Bbb{C}^{2},%
\text{ }\alpha ^{2}\right) $ to arbitrary $n$-dimensional representations $%
\left( \Bbb{C}^{n},\text{ }\alpha ^{n}\right) $ of $\Gamma _{N}^{2},$ for
all $n$ $\in $ $\Bbb{N}$ $\setminus $ $\{1\}.$

\strut To do that, we fix the matrices formed by

\begin{center}
$A_{2}(a,$ $b)$ $=$ $\left( 
\begin{array}{cc}
a & b \\ 
\frac{1-a^{2}}{b} & -a
\end{array}
\right) ,$
\end{center}

as blocks of certain matrices in $M_{n}(\Bbb{C}).$

\strut

\textbf{Assumption} For our main purpose, we always assume that:

\begin{center}
$(a,$ $b)$ $\in $ $\Bbb{C}_{\symbol{126}1}^{\times }$ $\times $ $\Bbb{C}%
^{\times },$
\end{center}

whenever we have matrices $A_{2}(a,$ $b)$ in the following text, as in
(2.1.3)$^{\prime }$, where

\begin{center}
$\strut \Bbb{C}_{\symbol{126}1}^{\times }$ $\overset{def}{=}$ $\Bbb{C}$ $%
\setminus $ $\{0,$ $1\}.$
\end{center}

$\square $

\strut

Assume first that $n$ $=$ $3,$ and consider a matrix $A_{3}$ formed by

(2.2.1)

\begin{center}
$A_{3}$ $=$ $\left( 
\begin{array}{ccc}
1 & c & \frac{-bc}{1-a} \\ 
0 & a & b \\ 
0 & \frac{1-a^{2}}{b} & -a
\end{array}
\right) $ $\in $ $M_{3}(\Bbb{C}),$
\end{center}

for some $a,$ $b,$ $c$ $\in $ $\Bbb{C}^{\times }.$ As we have seen in
(2.2.1), this matrix $A_{3}$ is regarded as the following block matrix,

(2.2.1)$^{\prime }$

\begin{center}
$A_{3}$ $=$ $\left( 
\begin{array}{cc}
(1) & \left( 
\begin{array}{ccc}
c &  & \frac{-bc}{1-a}
\end{array}
\right) \\ 
\left( 
\begin{array}{c}
0 \\ 
0
\end{array}
\right) & A_{2}(a,\text{ }b)
\end{array}
\right) $ in $M_{3}(\Bbb{C}),$
\end{center}

where $A_{2}$\strut $(a,$ $b)$ is in the sense of (2.1.1).

If we have $(3\times 3)$-matrix $A$ in the sense of (2.2.1) understood as
(2.2.1)$^{\prime },$ then, from straightforward computation, one can obtain
that

(2.2.2)

\begin{center}
$A_{3}^{2}$ $=$ $I_{3}$ $=$ $\left( 
\begin{array}{ccc}
1 & 0 & 0 \\ 
0 & 1 & 0 \\ 
0 & 0 & 1
\end{array}
\right) ,$
\end{center}

the identity matrix of $M_{3}(\Bbb{C}).$

\strut Let's define now a morphism

\begin{center}
$A_{3}$ $:$ $\Bbb{C}_{\symbol{126}1}^{\times }\times \Bbb{C}^{\times }\times 
\Bbb{C}^{\times }$ $\rightarrow $ $M_{3}(\Bbb{C})$ by
\end{center}

(2.2.3)

\begin{center}
$A_{3}(a,$ $b,$ $c)$ $=$ $\left( 
\begin{array}{cc}
\left( 1\right) & \left( 
\begin{array}{ccc}
c &  & \frac{-bc}{1-a}
\end{array}
\right) \\ 
\left( 
\begin{array}{c}
0 \\ 
0
\end{array}
\right) & A_{2}\left( a,\text{ }b\right)
\end{array}
\right) ,$
\end{center}

identified with the matrix (2.2.1) in $M_{3}(\Bbb{C}),$ by (2.2.1)$^{\prime
}.$

By (2.2.2), for any $(a,$ $b,$ $c)$ $\in $ $\Bbb{C}_{\symbol{126}1}^{\times
}\times \Bbb{C}^{\times }\times \Bbb{C}^{\times },$ we have

\begin{center}
$\left( A_{3}(a,b,c)\right) ^{2}$ $=$ $I_{3}$ in $M_{3}(\Bbb{C}).$
\end{center}

\strut \strut

Now, let $n$ $=$ $4$ in $\Bbb{N},$ and consider a matrix $A_{4}$ formed by

(2.2.3)

\begin{center}
$A_{4}$ $=$ $\left( 
\begin{array}{cccc}
1 & 0 & d & \frac{-bd}{1-a} \\ 
0 & 1 & c & \frac{-bc}{1-a} \\ 
0 & 0 & a & b \\ 
0 & 0 & \frac{1-a^{2}}{b} & -a
\end{array}
\right) $ $\in $ $M_{4}(\Bbb{C}),$
\end{center}

for $c,$ $d$ $\in $ $\Bbb{C}^{\times }.$ Then, similarly, this matrix $A_{4}$
of (2.2.3) is regarded as the block matrix,

(2.2.4)

\begin{center}
$A_{4}$ $=$ $\left( 
\begin{array}{cc}
(1) & \left( 
\begin{array}{ccc}
0 & d & \frac{-bd}{1-a}
\end{array}
\right) \\ 
\left( 
\begin{array}{c}
0 \\ 
0 \\ 
0
\end{array}
\right) & A_{3}(a,b,c)
\end{array}
\right) ,$
\end{center}

or

\begin{center}
$A_{4}$ $=$ $\left( 
\begin{array}{ccc}
I_{2} &  & \left( 
\begin{array}{ccc}
d &  & \frac{-bd}{1-a} \\ 
&  &  \\ 
c &  & \frac{-bc}{1-a}
\end{array}
\right) \\ 
&  &  \\ 
O_{2,2} &  & A_{2}(a,\text{ }b)
\end{array}
\right) .$
\end{center}

in $M_{4}(\Bbb{C}),$ where $O_{2,2}$ is the $(2\times 2)$-\emph{zero matrix}
whose all entries are zeroes.

\strut Then, by the direct computation, one obtains that

(2.2.5)

\begin{center}
$A_{4}^{2}$ $=$ $I_{4}$, the identity matrix of $M_{4}(\Bbb{C}).$
\end{center}

So, similar to (2.2.3), we define a morphism

\begin{center}
$A_{4}$ $:$ $\Bbb{C}_{\symbol{126}1}^{\times }\times \left( \Bbb{C}^{\times
}\right) ^{3}$ $\rightarrow $ $M_{4}(\Bbb{C})$ by
\end{center}

(2.2.6)

\begin{center}
$A_{4}(a,$ $b,$ $c,$ $d)$ $=$ $\left( 
\begin{array}{ccc}
I_{2} &  & Q_{a,b}\left( c,d\right) \\ 
&  &  \\ 
O_{2,2} &  & A_{2}(a,\text{ }b)
\end{array}
\right) ,$
\end{center}

where $I_{2}$ is the $(2\times 2)$-identity matrix, $O_{2,2}$ is the $%
(2\times 2)$-zero matrix, and

\begin{center}
$Q_{a,b}(c,d)$ $=$ $\left( 
\begin{array}{ccc}
d &  & \frac{-bd}{1-a} \\ 
&  &  \\ 
c &  & \frac{-bc}{1-a}
\end{array}
\right) $.
\end{center}

The image $A_{4}(a,b,c,d)$ of (2.2.6) becomes a well-defined matrix in $%
M_{4}(\Bbb{C}),$ by (2.2.3) and (2.2.4). Moreover, by (2.2.5), we have

(2.2.7)

\begin{center}
$\left( A_{4}(a,b,c,d)\right) ^{2}$ $=$ $I_{4}$ in $M_{4}(\Bbb{C}).$
\end{center}

\strut

Let's consider one more step: let $n$ $=$ $5.$ Similarly, we define the
following map,

\begin{center}
$A_{5}$ $:$ $\Bbb{C}_{\symbol{126}1}^{\times }\times \left( \Bbb{C}^{\times
}\right) ^{4}\rightarrow M_{5}(\Bbb{C})$ by
\end{center}

(2.2.8)

\begin{center}
$A_{5}\left( a_{1},\text{ ..., }a_{5}\right) $ $=$ $\left( 
\begin{array}{ccc}
I_{3} &  & Q_{a_{1},a_{2}}(a_{3},a_{4},a_{5}) \\ 
&  &  \\ 
O_{2,3} &  & A_{2}(a_{1},\text{ }a_{2})
\end{array}
\right) ,$
\end{center}

\strut

in $M_{5}(\Bbb{C}),$ for all $(a_{1},$ ..., $a_{5})$ $\in $ $\left( \Bbb{C}%
^{\times }\right) ^{5},$ where $A_{2}(a_{1},$ $a_{2})$ is in the sense of
(2.1.1), $O_{2,3}$ is the $(2\times 3)$-zero matrix, and $I_{3}$ is the $%
(3\times 3)$-identity matrix, and

\begin{center}
$Q_{a_{1},a_{2}}(a_{3},a_{4},a_{5})$ $=$ $\left( 
\begin{array}{cc}
a_{5} & \frac{-a_{2}a_{5}}{1-a_{1}} \\ 
a_{4} & \frac{-a_{2}a_{4}}{1-a_{1}} \\ 
a_{3} & \frac{-a_{2}a_{3}}{1-a_{1}}
\end{array}
\right) .$
\end{center}

i.e.,

\begin{center}
$A_{5}(a_{1},$ ..., $a_{5})$ $=$ $\left( 
\begin{array}{ccccc}
1 & 0 & 0 & a_{5} & \frac{-a_{2}a_{5}}{1-a_{1}} \\ 
0 & 1 & 0 & a_{4} & \frac{-a_{2}a_{4}}{1-a_{1}} \\ 
0 & 0 & 1 & a_{3} & \frac{-a_{2}a_{3}}{1-a_{1}} \\ 
0 & 0 & 0 & a_{1} & a_{2} \\ 
0 & 0 & 0 & \frac{1-a_{1}^{2}}{a_{2}} & -a_{1}
\end{array}
\right) $ $\in $ $M_{5}(\Bbb{C}).$
\end{center}

From direct computation, one again obtain that

(2.2.9)

\begin{center}
$\left( A_{5}(a_{1},...,a_{5})\right) ^{2}$ $=$ $I_{5},$ the identity matrix
of $M_{5}(\Bbb{C}).$
\end{center}

\strut

Inductively, for $n$ $\geq $ $3,$ we define the following map

\begin{center}
$A_{n}$ $:$ $\Bbb{C}_{\symbol{126}1}^{\times }\times \left( \Bbb{C}^{\times
}\right) ^{n-1}$ $\rightarrow $ $M_{n}(\Bbb{C}),$ by
\end{center}

(2.2.10)

\begin{center}
$A_{n}\left( a_{1},\text{ ..., }a_{n}\right) $ $=$ $\left( 
\begin{array}{ccc}
I_{n-2} &  & Q_{a_{1},a_{2}}\left( a_{3},...,a_{n}\right) \\ 
&  &  \\ 
O_{2,n-2} &  & A_{2}(a_{1},\text{ }a_{2})
\end{array}
\right) ,$
\end{center}

\strut

for all $(a_{1},$ ..., $a_{n})$ $\in $ $\left( \Bbb{C}^{\times }\right)
^{n}, $ where

\begin{center}
$A_{2}(a_{1},$ $a_{2})$ is in the sense of (2.1.1),

$O_{2,n-2}$ is the $(2\times (n-2))$-zero matrix,

$I_{n-2}$ is the $(n-2)\times (n-2)$ identity matrix,
\end{center}

and

\begin{center}
$Q_{a_{1},a_{2}}(a_{3},$ ..., $a_{n})$ $=$ $\left( 
\begin{array}{ccc}
\begin{array}{c}
a_{n}
\end{array}
&  & 
\begin{array}{c}
\frac{-a_{2}a_{n}}{1-a_{1}}
\end{array}
\\ 
\vdots &  & \vdots \\ 
\begin{array}{c}
a_{4}
\end{array}
&  & 
\begin{array}{c}
\frac{-a_{2}a_{3}}{1-a_{1}}
\end{array}
\\ 
a_{3} &  & \frac{-a_{2}a_{3}}{1-a_{1}}
\end{array}
\right) .$
\end{center}

Then we obtain the following computation.

\begin{theorem}
Let $A_{n}$ be the morphism (2.2.10), and let $A_{n}(a_{1},$ ..., $a_{n})$
be the image of $A_{n}$ realized in the matricial algebra $M_{n}(\Bbb{C}),$
for arbitrarily fixed $n$ $\geq $ $3$ in $\Bbb{N}.$ Then

(2.2.11)

\begin{center}
$\left( A_{n}(a_{1},...,a_{n})\right) ^{2}$ $=$ $I_{n},$
\end{center}

the identity matrix of $M_{n}(\Bbb{C}).$
\end{theorem}

\begin{proof}
Let $n$ $\geq 3$ be given in $\Bbb{N}.$ Then, by (2.2.2), (2.2.7) and
(2.2.9), we have

\begin{center}
$\left( A_{k}(a_{1},\text{ ..., }a_{k})\right) ^{2}$ $=$ $I_{k}$ in $M_{k}(%
\Bbb{C}),$
\end{center}

for $k$ $=$ $3,$ $4,$ $5.$

Now, without loss of generality take $n$ $\geq $ $6$ in $\Bbb{N},$
generally. Then, by (2.2.10),

\begin{center}
$A_{n}(a_{1},$ ..., $a_{n})$ $=$ $\left( 
\begin{array}{ccc}
I_{n-2} &  & Q_{a_{1},a_{2}}(a_{3},\text{ ..., }a_{n}) \\ 
&  &  \\ 
O_{2,n-2} &  & A_{2}(a_{1},a_{2})
\end{array}
\right) .$
\end{center}

For convenience, we let

\begin{center}
$I$ $\overset{denote}{=}$ $I_{n-2},\qquad $ $Q$ $\overset{denote}{=}$ $%
Q_{a_{1},a_{2}}(a_{3},$ ..., $a_{n}),$
\end{center}

and

\begin{center}
$O$ $\overset{denote}{=}$ $O_{2,n-2},\qquad $ $A$ $\overset{denote}{=}$ $%
A_{2}(a_{1},a_{2}).$
\end{center}

i.e.,

\begin{center}
$A_{n}(a_{1},$ ..., $a_{n})$ $\overset{denote}{=}$ $\left( 
\begin{array}{cc}
I & Q \\ 
O & A
\end{array}
\right) ,$
\end{center}

as a block matrix in $M_{n}(\Bbb{C}).$ Then

\strut

$\qquad \left( A_{n}(a_{1},\text{ ..., }a_{n})\right) ^{2}$ $=$ $\left( 
\begin{array}{cc}
I & Q \\ 
O & A
\end{array}
\right) \left( 
\begin{array}{cc}
I & Q \\ 
O & A
\end{array}
\right) $

\strut

$\qquad \qquad \qquad =$ $\left( 
\begin{array}{ccc}
I^{2}+QO &  & IQ+QA \\ 
&  &  \\ 
OI+AO &  & OQ+A^{2}
\end{array}
\right) $

\strut (2.2.12)

$\qquad \qquad \qquad =$ $\left( 
\begin{array}{ccc}
I &  & IQ+QA \\ 
&  &  \\ 
O &  & A^{2}
\end{array}
\right) $ $=$ $\left( 
\begin{array}{ccc}
I &  & Q+QA \\ 
&  &  \\ 
O &  & I_{2}
\end{array}
\right) ,$

\strut by (2.1.2).

So, to show (2.2.11), it is sufficient to show that

\begin{center}
$Q+QA$ $=$ $O_{n-2,2},$
\end{center}

by (2.2.12).

Notice that

$\qquad QA$ $=$ $\left( 
\begin{array}{cc}
a_{n} & \frac{-a_{2}a_{n}}{1-a_{1}} \\ 
\vdots & \vdots \\ 
a_{4} & \frac{-a_{2}a_{4}}{1-a_{1}} \\ 
a_{3} & \frac{-a_{2}a_{3}}{1-a_{1}}
\end{array}
\right) $ $\left( 
\begin{array}{cc}
a_{1} & a_{2} \\ 
\frac{1-a_{1}^{2}}{a_{2}} & -a_{1}
\end{array}
\right) $

\strut

$\qquad \qquad =$ $\left( 
\begin{array}{ccc}
a_{1}a_{n}+\frac{-a_{2}a_{n}(1-a_{1}^{2})}{a_{2}(1-a_{1})} &  & a_{2}a_{n}+%
\frac{a_{1}a_{2}a_{n}}{1-a_{1}} \\ 
\vdots &  & \vdots \\ 
a_{1}a_{3}+\frac{-a_{2}a_{3}(1-a_{1}^{2})}{a_{2}(1-a_{1})} &  & a_{2}a_{3}+%
\frac{a_{1}a_{2}a_{3}}{1-a_{1}}
\end{array}
\right) $

\strut

$\qquad \qquad =$ $\left( 
\begin{array}{ccc}
-a_{n} &  & a_{2}a_{n}\left( 1+\frac{a_{1}}{1-a_{1}}\right) \\ 
&  &  \\ 
\vdots &  &  \\ 
&  &  \\ 
-a_{3} &  & a_{2}a_{3}\left( 1+\frac{a_{1}}{1-a_{1}}\right)
\end{array}
\right) .$

\strut Thus,

$\qquad Q+QA$ $=$ $\left( 
\begin{array}{ccc}
a_{n} &  & \frac{-a_{2}a_{n}}{1-a_{1}} \\ 
\vdots &  & \vdots \\ 
a_{3} &  & \frac{-a_{2}a_{3}}{1-a_{1}}
\end{array}
\right) +\left( 
\begin{array}{ccc}
-a_{n} &  & a_{2}a_{n}\left( 1+\frac{a_{1}}{1-a_{1}}\right) \\ 
\vdots &  &  \\ 
-a_{3} &  & a_{2}a_{3}\left( 1+\frac{a_{1}}{1-a_{1}}\right)
\end{array}
\right) $

\strut

$\qquad \qquad =$ $\left( 
\begin{array}{ccc}
0 &  & a_{2}a_{n}\left( \frac{-1}{1-a_{1}}+1+\frac{a_{1}}{1-a_{1}}\right) \\ 
\vdots &  & \vdots \\ 
0 &  & a_{2}a_{3}\left( \frac{-1}{1-a_{1}}+1+\frac{a_{1}}{1-a_{1}}\right)
\end{array}
\right) $ $=$ $\left( 
\begin{array}{cc}
0 & 0 \\ 
\vdots & \vdots \\ 
0 & 0
\end{array}
\right) $

\strut

$\qquad \qquad =$ $O_{n-2,2},$

i.e.,

\begin{center}
$Q+QA$ $=$ $O_{n-2,2}.$
\end{center}

So, for any $n$ $\geq $ $6,$

(2.2.13)

\begin{center}
$\left( A_{n}(a_{1},\text{ ..., }a_{n})\right) ^{2}$ $=$ $\left( 
\begin{array}{ccc}
I &  & Q+QA \\ 
&  &  \\ 
O &  & I
\end{array}
\right) $ $=$ $\left( 
\begin{array}{cc}
I & O_{n-2,2} \\ 
O & I
\end{array}
\right) $ $=$ $I_{n},$
\end{center}

the identity matrix of $M_{n}(\Bbb{C}).$

\strut

Therefore, for any $n$ $\geq $ $3$ in $\Bbb{N},$

\begin{center}
\strut $\left( A_{n}(a_{1},\text{ ..., }a_{n})\right) ^{2}$ $=$ $I_{n}$ in $%
M_{n}(\Bbb{C}).$
\end{center}

\strut
\end{proof}

\subsection{$n$-Dimensional Hilbert-Space Representations of $\Gamma
_{N}^{2} $}

Let $N$ $\in $ $\Bbb{N}$ be the fixed quantity, and fix $n$ $\in $ $\Bbb{N}$ 
$\setminus $ $\{1\}.$ Also, let $\Gamma _{N}^{2}$ be our finitely presented
group (1.0.1). For the fixed $n$, we take the $n$-tuples

(2.3.1)

\begin{center}
$W_{k}$ $=$ $(a_{k,1},$ ..., $a_{k,n})$ $\in $ $\Bbb{C}_{\symbol{126}%
1}^{\times }\times \left( \Bbb{C}^{\times }\right) ^{n-1},$
\end{center}

for $k$ $=$ $1,$ ..., $N,$ and assume that $W_{1},$ ..., $W_{N}$ are
``mutually distinct.''\strut

\strut

\textbf{Observation and Notation} As we assumed above, let $W_{1},$ ..., $%
W_{N}$ be mutually distinct in $\Bbb{C}_{\symbol{126}1}^{\times }\times
\left( \Bbb{C}^{\times }\right) ^{n-1}.$ For our purposes, one may further
assume that these mutually distinct $n$-tuples satisfy

\begin{center}
$W_{i}\ncong $ $W_{j}$
\end{center}

in the sense that:

\begin{center}
$a_{i,l}$ $\neq $ $a_{j,l}$ in $\Bbb{C}^{\times },$ for all $l$ $=$ $1,$
..., $N.$
\end{center}

If the $n$-tuples $W_{1},$ ..., $W_{N}$ satisfy the above stronger condition
than the mutually-distinctness, we say they are \emph{strongly mutually
distinct}. $\square $

\strut \strut

Define the corresponding matrices,

(2.3.2)

\begin{center}
$X_{k}$ $\overset{denote}{=}$ $A_{n}\left( W_{k}\right) $ $=$ $\left( 
\begin{array}{ccc}
I_{n-2} &  & Q_{a_{k,1},a_{k,2}}\left( a_{k,3},\text{ ..., }a_{k,n}\right)
\\ 
&  &  \\ 
O_{2,\,n-2} &  & A_{2}(a_{k,1},\text{ }a_{k,2})
\end{array}
\right) ,$
\end{center}

\strut

as in (2.2.10), where $W_{k}$ are ``strongly mutually distinct'' $n$-tuples
of (2.3.1), for all $k$ $=$ $1,$ ..., $N.$

Then, as in Section 2.1, one can get the multiplicative subgroup $\frak{M}%
_{n}^{2}(N)$ of $M_{n}(\Bbb{C})$ by the (reduced) free group generated by $%
\{X_{1},$ ..., $X_{N}\}.$ i.e.,

$\strut $(2.3.3)

\begin{center}
$\frak{M}_{n}^{2}(N)$ $=$ $\left\langle \{X_{j}\}_{j=1}^{N}\right\rangle $
in $M_{n}(\Bbb{C}),$
\end{center}

where $X_{j}$'s are in the sense of (2.3.2), for all $j$ $=$ $1,$ ..., $N.$

\begin{definition}
We call the multiplicative subgroup $\frak{M}_{n}^{2}(N)$ of $M_{n}(\Bbb{C}%
), $ an $n$-dimensional, order-2, $N$-generator (sub-)group (of $M_{n}(\Bbb{C%
})$).
\end{definition}

\strut Then, similar to the proof of the structure theorem (2.1.7), one can
obtain the following generalized result.

\begin{theorem}
Let $\frak{M}_{n}^{2}(N)$ be an $n$-dimensional, order-2, $N$-generator
group in $M_{n}(\Bbb{C}).$ Then it is group-isomorphic to the group $\Gamma
_{N}^{2}$ of (1.0.1). i.e.,

(2.3.4)

\begin{center}
$\frak{M}_{n}^{2}(N)$ $\overset{\text{Group}}{=}$ $\Gamma _{N}^{2},$ for all 
$n$ $\in $ $\Bbb{N}$ $\setminus $ $\{1\}.$
\end{center}

In particular, there exist generator-preserving group-isomorphisms $\alpha
^{n}:\Gamma _{N}^{2}\rightarrow \frak{M}_{n}^{2}(N),$ such that

(2.3.5)

\begin{center}
$\alpha ^{n}\left( x_{j}\right) $ $=$ $X_{j},$ for all $j$ $=$ $1,$ ..., $N,$
\end{center}

where $x_{j}$'s are the generators of $\Gamma _{N}^{2},$ and $X_{j}$'s are
in the sense of (2.3.2).
\end{theorem}

\begin{proof}
Let $n$ $=$ $2.$ Then, by (2.1.7), the isomorphic relation (2.3.4) holds.
Assume now that $n$ $\geq 3.$ Then the proof of (2.3.4) is similar to that
of (2.1.7). Indeed, one can show that

\begin{center}
$\frak{M}_{n}^{2}(N)$ $\overset{\text{Group}}{=}$ $\mathcal{F}_{N}$ $/$ $%
\mathcal{R}_{N}$ $\overset{\text{Group}}{=}$ $\Gamma _{N}^{2},$
\end{center}

for all $n$ $\in $ $\Bbb{N}.$ Therefore, there exists a natural
generator-preserving group-isomorphism $\alpha ^{n}$ from $\Gamma _{N}^{2}$
onto $\frak{M}_{n}^{2}(N)$ as in (2.3.5).
\end{proof}

\strut The structure theorem (2.3.4) is the generalized result of (2.1.7).
And the group-isomorphism $\alpha ^{n}$ of (2.3.5) generalizes $\alpha ^{2}$
of (2.1.8). Therefore, we obtain the following theorem, generalizing (2.1.9).

\begin{theorem}
Let $\Gamma $\strut $_{N}^{2}$ be the group in the sense of (1.0.1), for
some $N$ $\in $ $\Bbb{N}.$ Then there exist $n$-dimensional Hilbert-space
representation $\left( \Bbb{C}^{n},\text{ }\alpha ^{n}\right) $ of $\Gamma
_{N}^{2},$ where $\alpha ^{n}$ are in the sense of (2.3.5), for all $n$ $\in 
$ $\Bbb{N}$ $\setminus $ $\{1\}.$ i.e.,

(2.3.6)\qquad $\left( \Bbb{C}^{n},\text{ }\alpha ^{n}\right) $ form $n$%
-dimensional representations of $\Gamma _{N}^{2},$

for all $n$ $\in $ $\Bbb{N}.$
\end{theorem}

\begin{proof}
\strut If $n$ $=$ $2,$ as we have seen in (2.1.9), there exists a
2-dimensional Hilbert-space representation $\left( \Bbb{C}^{2},\text{ }%
\alpha ^{2}\right) $, where $\alpha ^{2}$ is the group-isomorphism in the
sense of (2.1.8).

Suppose $n$ $\geq $ $3$ in $\Bbb{N}.$ For such $n,$ two groups $\Gamma
_{N}^{2}$ and $\frak{M}_{n}^{2}(N)$ are isomorphic with a group-isomorphism $%
\alpha ^{n}$ $:$ $\Gamma _{N}^{2}$ $\rightarrow $ $\frak{M}_{n}^{2}(N)$ of
(2.3.5), by (2.3.4). It shows that the images $\alpha ^{n}(g)$ are $(n\times
n)$-matrices acting on the $n$-dimensional Hilbert space $\Bbb{C}^{n}.$ So,
the pair $\left( \Bbb{C}^{n},\text{ }\alpha ^{n}\right) $ forms an $n$%
-dimensional representation of $\Gamma _{N}^{2},$ for all $n$ $\in $ $\Bbb{N}%
.$
\end{proof}

\strut The above two theorems show that, for a fixed group $\Gamma _{N}^{2}$
of (1.0.1), one has a system

\begin{center}
$\left\{ \left( \Bbb{C}^{n},\text{ }\alpha ^{n}\right) \right\}
_{n=2}^{\infty }$
\end{center}

of Hilbert-space representations, and the corresponding isomorphic groups

\begin{center}
$\left\{ \frak{M}_{n}^{2}(N)\right\} _{n=2}^{\infty },$
\end{center}

acting on $\Bbb{C}^{n}$ (or, acting in $M_{n}(\Bbb{C})$).

\begin{example}
Let $\Gamma _{2}^{2}$ be the finitely presented group,

\begin{center}
$\Gamma _{2}^{2}$ $=$ $\left\langle \{x_{1},\text{ }x_{2}\},\text{ }%
\{x_{1}^{2}=x_{2}^{2}\}\right\rangle .$
\end{center}

Fix $n$ $=$ $3.$ Now, take the following strongly distinct triples,

\begin{center}
$W_{1}$ $=$ $(t_{1},$ $t_{2},$ $t_{3}),$ $W_{2}$ $=$ $(s_{1},$ $s_{2},$ $%
s_{3}),$
\end{center}

in $\Bbb{C}_{\symbol{126}1}^{\times }\times \left( \Bbb{C}^{\times }\right)
^{2},$ and construct two matrices,

\strut

$A_{3}(W_{1})$ $=$ $\left( 
\begin{array}{ccc}
1 & t_{3} & \frac{-t_{2}t_{3}}{1-t_{1}} \\ 
0 & t_{1} & t_{2} \\ 
0 & \frac{1-t_{1}^{2}}{t_{2}} & -t_{1}
\end{array}
\right) ,$ $A_{3}(W_{2})$ $=$ $\left( 
\begin{array}{ccc}
1 & s_{3} & \frac{-s_{2}s_{3}}{1-s_{1}} \\ 
0 & s_{1} & s_{2} \\ 
0 & \frac{1-s_{1}^{2}}{s_{2}} & -s_{1}
\end{array}
\right) ,$

\strut

in $M_{3}(\Bbb{C}).$ The group $\frak{M}_{3}^{2}(2)$ is established as the
reduced free group

\begin{center}
$\left\langle \{A_{3}(W_{1}),\text{ }A_{3}(W_{2})\}\right\rangle $
\end{center}

generated by $A_{3}(W_{1})$ and $A_{3}(W_{2})$ in $M_{3}(\Bbb{C}).$Then

\begin{center}
$\frak{M}_{3}^{2}(2)$ $\overset{\text{Group}}{=}$ $\Gamma _{2}^{2}.$
\end{center}

So, one has a natural $3$-dimensional representation $\left( \Bbb{C}^{3},%
\text{ }\alpha ^{3}\right) ,$ where $\alpha ^{3}$ is the group-isomorphism
satisfying

\begin{center}
$\alpha ^{3}\left( x_{l}\right) $ $=$ $A_{3}(W_{l}),$ for all $l$ $=$ $1,$ $%
2.$
\end{center}

\strut
\end{example}

\section{\strut Application: A Group Algebra $\mathcal{A}_{2}$ Induced by $%
\Gamma _{2}^{2}$}

In Section 2, we showed that the finitely presented group $\Gamma _{N}^{2}$
of (1.0.1) can have its family of finite-dimensional Hilbert-space
representations

\begin{center}
$\left\{ \left( \Bbb{C}^{n},\text{ }\alpha ^{n}\right) \right\}
_{n=2}^{\infty },$
\end{center}

since it is group-isomorphic to $n$-dimensional, order-2, $N$-generator
groups $\frak{M}_{n}^{2}(N)$ in $M_{n}(\Bbb{C}),$ for all $n$ $\in $ $\Bbb{N}
$ $\setminus $ $\{1\},$ where

\begin{center}
$\frak{M}_{n}^{2}(N)$ $=$ $\left\langle
\{A_{n}(W_{k})\}_{k=1}^{N}\right\rangle $ in $M_{n}(\Bbb{C}),$
\end{center}

where $W_{1},$ ..., $W_{N}$ are strongly mutually distinct $n$-tuples of $%
\Bbb{C}_{\symbol{126}1}^{\times }\times \left( \Bbb{C}^{\times }\right)
^{n-1},$ for all $n$ $\in $ $\Bbb{N}$ $\setminus $ $\{1\}.$

\subsection{Algebraic Observation for Group Elements of $\Gamma _{2}^{2}$}

In this section, we concentrate on the case where $N$ $=$ $2,$ i.e., we
study the group elements of $\Gamma _{2}^{2}$ in detail. As a special case
of (2.3.4),

(3.1.1)

\begin{center}
$\left\langle \{x_{1},x_{2}\},\text{ }\{x_{1}^{2}=x_{2}^{2}=e_{2}\}\right%
\rangle =\Gamma _{2}^{2}$ $\overset{\text{Group}}{=}$ $\frak{M}_{2}^{2}(2).$
\end{center}

Consider $\Gamma _{2}^{2}$ pure-algebraically. Each element $g$ of $\Gamma
_{2}^{2}$ has its expression,

(3.1.2)

\begin{center}
$g$ $=$ $x_{i_{1}}^{k_{1}}x_{i_{2}}^{k_{2}}\cdot \cdot \cdot
x_{i_{n}}^{k_{i_{n}}},$ for some $n$ $\in $ $\Bbb{N},$
\end{center}

as in (1.0.3), for some $(i_{1},$ ..., $i_{n})$ $\in $ $\{1,$ $2\}^{n},$ and 
$(k_{1},$ ..., $k_{n})$ $\in $ $\Bbb{Z}^{n}.$

However, by the relation on $\Gamma _{2}^{2},$

\begin{center}
$x_{1}^{2}$ $=$ $e$ $\overset{denote}{=}$ $e_{2}$ $=$ $x_{2}^{2}$ in $\Gamma
_{2}^{2},$
\end{center}

one has that

(3.1.3)

\begin{center}
$x_{l}^{-1}$ $=$ $x_{l},$ for $l$ $=$ $1,$ $2,$
\end{center}

\strut and

\begin{center}
$x_{l}^{2k+1}$ $=$ $x_{l},$ for all $k$ $\in $ $\Bbb{N}.$
\end{center}

In other words, the generators $x_{1}$ and $x_{2}$ are ``self-invertible,''
and ``indeed of order-2,'' in $\Gamma _{2}^{2}.$ The above two conditions in
(3.1.3) can be summarized by

(3.1.3)$^{\prime }$

\begin{center}
$x_{l}^{2n+1}$ $=$ $x_{l},$ for $l$ $=$ $1,$ $2,$ for all $n$ $\in $ $\Bbb{Z}.$
\end{center}

\strut So, the general expression (3.1.2) of $g$ is in fact

(3.1.4)

\begin{center}
$g$ $=$ $x_{j_{1}}x_{j_{2}}...x_{j_{n}}$ in $\Gamma _{2}^{2},$
\end{center}

for some $(j_{1},$ ..., $j_{n})$ $\in $ $\{1,$ $2\}^{n},$ for some $n$ $\in $
$\Bbb{N},$ by (3.1.3), or by (3.1.3)$^{\prime }.$

By the characterization,

\begin{center}
$\Gamma _{2}^{2}$ $\overset{\text{Group}}{=}$ $\mathcal{F}_{2}$ $/$ $%
\mathcal{R}_{2},$
\end{center}

and by the definition of the noncommutative (non-reduced) free group $%
\mathcal{F}_{2},$ the expression (3.1.4) of $g$ goes to

(3.1.5)

\begin{center}
$g$ $=$ $\left\{ 
\begin{array}{lll}
e=e_{2} &  & \text{or} \\ 
x_{1} &  & \text{or} \\ 
x_{2} &  & \text{or} \\ 
(x_{1}x_{2})^{n}x_{1} &  & \text{or} \\ 
\left( x_{1}x_{2}\right) ^{n} &  & \text{or} \\ 
\left( x_{2}x_{1}\right) ^{n}x_{2} &  & \text{or} \\ 
\left( x_{2}x_{1}\right) ^{n}, &  & 
\end{array}
\right. $
\end{center}

in $\Gamma _{2}^{2},$ for all $n$ $\in $ $\Bbb{N}.$

\begin{proposition}
Let $g$ $\in $ $\Gamma _{2}^{2}.$ Then $g$ is only one of the forms in
(3.1.5). $\square $
\end{proposition}

\strut Let $g$ $\in $ $\Gamma _{2}^{2}$ $\setminus $ $\{e,$ $x_{1},$ $%
x_{2}\}.$ Say,

\begin{center}
$g$ $=$ $x_{1}x_{2}x_{1}x_{2}...x_{1}x_{2}x_{1}.$
\end{center}

Then it is self-invertible. Indeed,

$\qquad \qquad g^{2}$ $=$ $\left( x_{1}x_{2}...x_{1}x_{2}x_{1}\right) \left(
x_{1}x_{2}...x_{1}x_{2}x_{1}\right) $

$\qquad \qquad \qquad =$ $x_{1}x_{2}...x_{1}x_{2}\left( x_{1}^{2}\right)
x_{2}...x_{1}x_{2}x_{1}$

$\qquad \qquad \qquad =$ $x_{1}x_{2}...x_{1}x_{2}ex_{2}...x_{1}x_{2}x_{1}$

$\qquad \qquad \qquad =$ $x_{1}x_{2}...x_{1}\left( x_{2}^{2}\right)
x_{1}...x_{1}x_{2}x_{1}$

$\qquad \qquad \qquad =$ $...$

$\qquad \qquad \qquad =$ $x_{1}^{2}$ $=$ $e,$

i.e.,

(3.1.6)

\begin{center}
$\left( (x_{1}x_{2})^{n}x_{1}\right) ^{2}$ $=$ $e,$ for all $n$ $\in $ $\Bbb{%
N}.$
\end{center}

Similarly, one obtains that

(3.1.7)

\begin{center}
$\left( (x_{2}x_{1})^{n}x_{2}\right) ^{2}$ $=$ $e$ in $\Gamma _{2}^{2},$
\end{center}

\strut for all $n$ $\in $ $\Bbb{N}.$

Now, let

\begin{center}
$g$ $=$ $x_{1}x_{2}x_{1}x_{2}...x_{1}x_{2}$ in $\Gamma _{2}^{2}.$
\end{center}

Then

\begin{center}
$g^{-1}$ $=$ $x_{2}x_{1}x_{2}x_{1}...x_{2}x_{1},$
\end{center}

with

\begin{center}
$\left| g\right| $ $=$ $\left| g^{-1}\right| $ in $\Bbb{N},$
\end{center}

where $\left| w\right| $ means the\emph{\ length of} $w$ $\in $ $\Gamma
_{2}^{2}.$ (For example, if $g$ $=$ $x_{1}x_{2}x_{1}x_{2}x_{1}$ $\in $ $%
\Gamma _{2}^{2},$ then $\left| g\right| $ $=$ $5.$)

Indeed,

\begin{center}
$gg^{-1}$ $=$ $\left( x_{1}x_{2}...x_{1}x_{2}\right) \left(
x_{2}x_{1}...x_{2}x_{1}\right) $ $=$ $e.$
\end{center}

So, one has that

(3.1.8)

\begin{center}
$\left( (x_{1}x_{2})^{n}\right) ^{-1}$ $=$ $(x_{2}x_{1})^{n},$
\end{center}

and equivalently,

(3.1.9)

\begin{center}
$\left( (x_{2}x_{1})^{n}\right) ^{-1}$ $=$ $(x_{1}x_{2})^{n},$
\end{center}

for all $n$ $\in $ $\Bbb{N}$.

\begin{proposition}
Let $g$ $\in $ $\Gamma _{2}^{2}$ $\setminus $ $\{e\},$ and let $\left|
g\right| $ means the length of $g$ in $\{x_{1},$ $x_{2}\}$ in $\Gamma
_{2}^{2}.$

(3.1.10) If $\left| g\right| $ is odd in $\Bbb{N},$ then $g^{-1}=g$ in $%
\Gamma _{2}^{2}.$

(3.1.11) If $\left| g\right| $ is even in $\Bbb{N},$ then

\begin{center}
$g$ $=$ $(x_{1}x_{2})^{k},$ or $\left( x_{2}x_{1}\right) ^{k},$
\end{center}

and

\begin{center}
$g^{-1}$ $=$ $(x_{2}x_{1})^{k},$ respectively, $(x_{1}x_{2})^{k},$
\end{center}

for some $k$ $\in $ $\Bbb{N}.$
\end{proposition}

\begin{proof}
Suppose $\left| g\right| $ is odd in $\Bbb{N}.$ Then, by (3.1.5), the group
element $g$ is one of

\begin{center}
$x_{1},$ or $x_{2},$ or $(x_{1}x_{2})^{n}x_{1},$ or $(x_{2}x_{1})^{n}x_{2},$
\end{center}

for $n$ $\in $ $\Bbb{N}.$ By (3.1.6) and (3.1.7), in such cases,

\begin{center}
$g^{2}$ $=$ $e$ in $\Gamma _{2}^{2}.$
\end{center}

So, the statement (3.1.10) holds.

\strut

Now, assume that $\left| g\right| $\strut is even in $\Bbb{N}.$ Then, by
(3.1.5), this element $g$ is either

\begin{center}
$\left( x_{1}x_{2}\right) ^{k},$ or $(x_{2}x_{1})^{k},$ for $k$ $\in $ $\Bbb{%
N}.$
\end{center}

By (3.1.8) and (3.1.9), one has that

\begin{center}
$\left( \left( x_{1}x_{2}\right) ^{k}\right) ^{-1}$ $=$ $(x_{2}x_{1})^{k},$
\end{center}

respectively,

\begin{center}
$\left( (x_{2}x_{1})^{k}\right) ^{-1}$ $=$ $(x_{1}x_{2})^{k},$
\end{center}

for all $k$ $\in $ $\Bbb{N}.$

Therefore, the statement (3.1.11) holds, too.
\end{proof}

\strut By the self-invertibility of the group-identity $e,$ the generators $%
x_{1}$, $x_{2},$ and the group elements $g$ with odd length $\left| g\right| 
$, we are interested in the cases where

\begin{center}
$\left| g\right| $ is even.
\end{center}

Equivalently, we are interested in the cases where

\begin{center}
$g$ $=$ $\left( x_{1}x_{2}\right) ^{k}$, or $g$ $=$ $\left(
x_{2}x_{1}\right) ^{k}$ in $\Gamma _{2}^{2}.$
\end{center}

\subsection{\strut Analytic-and-Combinatorial Observation for Elements of $%
\Gamma _{2}^{2}$ in $M_{2}(\Bbb{C})$}

As we have seen, the group $\Gamma _{2}^{2}$ is group-isomorphic to the $2$%
-dimensional, order-2, $2$-generator subgroup $\frak{M}_{2}^{2}(2)$ of $%
M_{2}(\Bbb{C})$ under our representation $\left( \Bbb{C}^{2},\text{ }\alpha
^{2}\right) .$ In particular, $\frak{M}_{2}^{2}(2)$ is generated by two
matrices

\begin{center}
$A_{2}(t_{1},$ $s_{1}),$ and $A_{2}(t_{2},$ $s_{2})$
\end{center}

\strut for the strongly mutually distinct pairs $(t_{1},$ $s_{1})$ and $%
(t_{2},$ $s_{2})$ in $\Bbb{C}^{\times }\times \Bbb{C}^{\times }.$ More
precisely,

\begin{center}
$A_{2}(t_{1},$ $s_{1})$ $=$ $\left( 
\begin{array}{cc}
t_{1} & s_{1} \\ 
\frac{1-t_{1}^{2}}{s_{1}} & -t_{1}
\end{array}
\right) ,$ and $A_{2}(t_{2},$ $s_{2})$ $=$ $\left( 
\begin{array}{cc}
t_{2} & s_{2} \\ 
\frac{1-t_{2}^{2}}{s_{2}} & -t_{2}
\end{array}
\right) $
\end{center}

in $M_{2}(\Bbb{C}).$

\strut

\textbf{Notation} In the rest of this paper, we denote $A_{2}(t_{1},$ $%
s_{1}) $ and $A_{2}(t_{2},$ $s_{2}),$ by $X_{1}$ and $X_{2},$ respectively,
for convenience. $\square $

\strut \strut \strut

As we mentioned at the end of Section 3.1, we are interested in
group-elements $g$ formed by

(3.2.1)

\begin{center}
$g$ $=$ $\left( x_{1}x_{2}\right) ^{k}$ $\in $ $\Gamma _{2}^{2},$ for $k$ $%
\in $ $\Bbb{N}.$
\end{center}

i.e., under $\alpha ^{2},$ if $g$ is not of the form of (3.2.1), we obtain a
self-invertible matrix

\begin{center}
$\alpha ^{2}(g),$ such that $\left( \alpha ^{2}(g)\right) ^{2}$ $=$ $I_{2}$
\end{center}

in $\frak{M}_{2}^{2}(2)$ $\subset $ $M_{2}(\Bbb{C}).$

\strut However, if $g$ is of (3.2.1) in $\Gamma _{2}^{2},$ then

(3.2.2)

\begin{center}
$\alpha ^{2}\left( g\right) $ $=$ $\alpha ^{2}\left( (x_{1}x_{2})^{k}\right) 
$ $=$ $\left( X_{1}X_{2}\right) ^{k},$
\end{center}

and hence,

\begin{center}
$\left( \alpha ^{2}(g)\right) ^{-1}$ $=$ $\alpha ^{2}\left( g^{-1}\right) $ $%
=$ $\left( X_{2}X_{1}\right) ^{k},$
\end{center}

in $\frak{M}_{2}^{2}(2).$

Observe that

(3.2.3)

\begin{center}
$\left( X_{1}X_{2}\right) +\left( X_{2}X_{1}\right) $ $=$ $\left(
2t_{1}t_{2}+\frac{s_{1}(1-t_{2}^{2})}{s_{2}}+\frac{s_{2}(1-t_{1}^{2})}{s_{1}}%
\right) I_{2}$
\end{center}

from the straightforward computation in $M_{2}(\Bbb{C}).$

Consider now that

\begin{center}
$\det \left( X_{1}\right) $ $=$ $-1$ $=$ $\det (X_{2}),$
\end{center}

and

$\qquad \det \left( X_{1}+X_{2}\right) $ $=$ $\det \left( 
\begin{array}{cc}
t_{1}+t_{2} & s_{1}+s_{2} \\ 
\frac{1-t_{1}^{2}}{s_{1}}+\frac{1-t_{2}^{2}}{s_{2}} & -t_{1}-t_{2}
\end{array}
\right) $

\strut

$\qquad =$ $-\left( t_{1}^{2}+2t_{1}t_{2}+t_{2}^{2}\right) -\left(
1-t_{1}^{2}+\frac{s_{1}\left( 1-t_{2}^{2}\right) }{s_{2}}+\frac{s_{2}\left(
1-t_{1}^{2}\right) }{s_{1}}+1-t_{2}^{2}\right) $

\strut

$\qquad =$ $-t_{1}^{2}-2t_{1}t_{2}-t_{2}^{2}-1+t_{1}^{2}-\frac{s_{1}\left(
1-t_{2}^{2}\right) }{s_{2}}-\frac{s_{2}\left( 1-t_{1}^{2}\right) }{s_{1}}%
-1+t_{2}^{2}$

\strut

$\qquad =$ $-2t_{1}t_{2}-\frac{s_{1}\left( 1-t_{2}^{2}\right) }{s_{2}}-\frac{%
s_{2}\left( 1-t_{1}^{2}\right) }{s_{1}}-2,$

i.e.

\begin{center}
$\det \left( X_{1}+X_{2}\right) $ $=$ $-2-\left( 2t_{1}t_{2}+\frac{%
s_{1}\left( 1-t_{2}^{2}\right) }{s_{2}}+\frac{s_{2}\left( 1-t_{1}^{2}\right) 
}{s_{1}}\right) ,$
\end{center}

\strut where $\det (\cdot )$ means the \emph{determinant on} $M_{2}(\Bbb{C})$

So, one can get that

(3.2.4)

\begin{center}
$\det \left( X_{1}\right) +\det \left( X_{2}\right) -\det \left(
X_{1}+X_{2}\right) $ $=$ $2t_{1}t_{2}+\frac{s_{1}\left( 1-t_{2}^{2}\right) }{%
s_{2}}+\frac{s_{2}\left( 1-t_{1}^{2}\right) }{s_{1}}.$
\end{center}

\begin{lemma}
Let $X_{1}$ and $X_{2}$ be the generating matrices of $\frak{M}_{2}^{2}(2)$
as above. Then there exists $\varepsilon _{0}$ $\in $ $\Bbb{C},$ such that

(3.2.5)

\begin{center}
$\left( X_{1}X_{2}\right) $ $+$ $(X_{2}X_{1})$ $=$ $\varepsilon _{0}$ $I_{2}$
in $M_{2}(\Bbb{C}),$
\end{center}

and

\begin{center}
$
\begin{array}{ll}
\varepsilon _{0} & =2t_{1}t_{2}+\frac{s_{1}\left( 1-t_{2}^{2}\right) }{s_{2}}%
+\frac{s_{2}\left( 1-t_{1}^{2}\right) }{s_{1}} \\ 
&  \\ 
& =\det \left( X_{1}\right) +\det \left( X_{2}\right) -\det \left(
X_{1}+X_{2}\right) .
\end{array}
$
\end{center}
\end{lemma}

\begin{proof}
The proof of (3.2.5) is done by (3.2.3) and (3.2.4).
\end{proof}

\strut Therefore, by (3.2.5), one can observe that

$\qquad \left( X_{1}X_{2}+X_{2}X_{1}\right) ^{2}=$ $\left( X_{1}X_{2}\right)
^{2}+X_{1}X_{2}X_{2}X_{1}$

$\qquad \qquad \qquad \qquad \qquad \qquad \qquad \quad
+X_{2}X_{1}X_{1}X_{2}+\left( X_{2}X_{1}\right) ^{2}$

$\qquad \qquad \qquad \qquad \qquad =$ $\left( X_{1}X_{2}\right) ^{2}+\left(
X_{2}X_{1}\right) ^{2}+2I_{2},$

and hence,

\begin{center}
$
\begin{array}{ll}
\left( X_{1}X_{2}\right) ^{2}+\left( X_{2}X_{1}\right) ^{2} & =\left(
X_{1}X_{2}+X_{2}X_{1}\right) ^{2}-2I_{2} \\ 
& =\left( \varepsilon _{0}I_{2}\right) ^{2}-2I_{2}=\left( \varepsilon
_{0}^{2}-2\right) I_{2},
\end{array}
$
\end{center}

where $\varepsilon _{0}$ is in the sense of (3.2.5). i.e., we obtain that

(3.2.6)

\begin{center}
$\left( X_{1}X_{2}\right) ^{2}+\left( X_{2}X_{1}\right) ^{2}$ $=$ $\left(
\varepsilon _{0}^{2}-2\right) I_{2}.$
\end{center}

By (3.2.5) and (3.2.6), we have that

\begin{center}
$(X_{1}X_{2})^{1}+(X_{2}X_{1})^{1}$ $=$ $\varepsilon _{0}I_{2},$
\end{center}

and

\begin{center}
$\left( X_{1}X_{2}\right) ^{2}+(X_{2}X_{1})^{2}$ $=$ $\left( \varepsilon
_{0}^{2}-2\right) I_{2},$
\end{center}

where

\begin{center}
$\varepsilon _{0}$ $=$ $\det (X_{1})+\det (X_{2})-\det \left(
X_{1}+X_{2}\right) ,$
\end{center}

as in (3.2.5).

More generally, we obtain the following recurrence relation.

\begin{theorem}
Let $X_{1}$ and $X_{2}$ be the generating matrices of $\frak{M}_{2}^{2}(2)$
in $M_{2}(\Bbb{C}).$ If we denote

\begin{center}
$\frak{X}_{k}$ $\overset{denote}{=}$ $\left( X_{1}X_{2}\right)
^{k}+(X_{2}X_{1})^{k},$ for all $k$ $\in $ $\Bbb{N},$
\end{center}

then the following recurrence relation is obtained:

(3.2.7)

\begin{center}
$\frak{X}_{1}$ $=$ $\varepsilon _{0}$ $I_{2},$ and $\frak{X}_{2}$ $=$ $%
\left( \varepsilon _{0}^{2}-2\right) $ $I_{2},$
\end{center}

and

\begin{center}
$\frak{X}_{n}$ $=$ $\varepsilon _{0}$ $\frak{X}_{n-1}$ $-$ $\frak{X}_{n-2},$
for all $n$ $\geq $ $3$ in $\Bbb{N},$
\end{center}

in $M_{2}(\Bbb{C}),$ where

\begin{center}
$\varepsilon _{0}$ $=$ $\det \left( X_{1}\right) +\det (X_{2})-\det
(X_{1}+X_{2}).$
\end{center}
\end{theorem}

\begin{proof}
By (3.2.5), indeed, one has

\begin{center}
$\frak{X}_{1}$ $=$ $\varepsilon _{0}$ $I_{2},$
\end{center}

and by (3.2.6),

\begin{center}
$\frak{X}_{2}$ $=$ $\left( \varepsilon _{0}^{2}-2\right) $ $I_{2}.$
\end{center}

\strut

Suppose $n$ $=$ $3$ in $\Bbb{N}.$ Then

$\qquad \qquad \frak{X}_{3}$ $=$ $\left( X_{1}X_{2}\right) ^{3}+\left(
X_{2}X_{1}\right) ^{3}$

$\qquad \qquad \qquad =$ $\left( X_{1}X_{2}+X_{2}X_{1}\right) ^{3}-\left(
3X_{1}X_{2}+3X_{2}X_{1}\right) $

by (3.1.5)

$\qquad \qquad \qquad =$ $\left( \frak{X}_{1}\right) ^{3}-3\frak{X}_{1}$ $=$ 
$\left( \varepsilon _{0}I_{2}\right) ^{3}-3\left( \varepsilon
_{0}I_{2}\right) $

$\qquad \qquad \qquad =$ $\left( \varepsilon _{0}^{3}-3\varepsilon
_{0}\right) I_{2}$

in $M_{2}(\Bbb{C}).$

Observe now that

$\qquad \varepsilon _{0}\frak{X}_{2}-\frak{X}_{1}$ $=$ $\varepsilon
_{0}\left( \varepsilon _{0}^{2}-2\right) I_{2}-\varepsilon _{0}I_{2}$

$\qquad \qquad \qquad \qquad =$ $\left( \varepsilon _{0}^{3}-2\varepsilon
_{0}\right) I_{2}-\varepsilon _{0}I_{2}$

$\qquad \qquad \qquad \qquad =$ $\left( \varepsilon _{0}^{3}-3\varepsilon
_{0}\right) I_{2},$

in $M_{2}(\Bbb{C}).$

Thus, one obtains that

(3.2.8)

\begin{center}
$\frak{X}_{3}$ $=$ $\left( \varepsilon _{0}^{3}-3\varepsilon _{0}\right)
I_{2}$ $=$ $\varepsilon _{0}\frak{X}_{2}-\frak{X}_{1}.$
\end{center}

So, if $n$ $=$ $3,$ then the relation (3.2.7) holds true.

\strut

Assume now that the statement

(3.2.9)

\begin{center}
$\frak{X}_{n_{0}}$ $=$ $\varepsilon _{0}\frak{X}_{n_{0}-1}-\frak{X}%
_{n_{0}-2} $
\end{center}

hold for a fixed $n_{0}$ $\geq 3$ in $\Bbb{N}.$ Then

$\qquad \frak{X}_{n_{0}+1}$ $=$ $\left(
(X_{1}X_{2})^{n_{0}}+(X_{2}X_{1})^{n_{0}}\right) \left(
X_{1}X_{2}+X_{2}X_{1}\right) $

$\qquad \qquad \qquad \qquad \qquad -\left( X_{1}X_{2}\right) ^{n_{0}}\left(
X_{2}X_{1}\right) -\left( X_{2}X_{1}\right) ^{n_{0}}\left( X_{1}X_{2}\right) 
$

$\qquad \qquad \qquad =$ $\left(
(X_{1}X_{2})^{n_{0}}+(X_{2}X_{1})^{n_{0}}\right) \left( \varepsilon
_{0}I\right) $

$\qquad \qquad \qquad \qquad \qquad -$ $\left(
(X_{1}X_{2})^{n_{0}-1}+(X_{2}X_{1})^{n_{0}-1}\right) $

by the assumption (3.2.9)

$\qquad \qquad \qquad =$ $\varepsilon _{0}\frak{X}_{n_{0}}-\frak{X}%
_{n_{0}-1},$

i.e., for the fixed $n_{0}$ $\geq $ $3,$

(3.2.10)

\begin{center}
$\frak{X}_{n_{0}+1}$ $=$ $\varepsilon _{0}\frak{X}_{n_{0}}-\frak{X}%
_{n_{0}-1}.$
\end{center}

Since $n_{0}$ is arbitrary in $\Bbb{N}$ $\setminus $ $\{1,$ $2\},$ we can
conclude that

\begin{center}
$\frak{X}_{n+1}=\varepsilon \frak{X}_{n}-\frak{X}_{n-1}$,
\end{center}

for all $n$ $\in $ $\Bbb{N}$ $\setminus $ $\{1\},$ by the induction, because
of (3.2.8), (3.2.9) and (3.2.10).

\strut

Therefore, we obtain the recurrence relation:

\begin{center}
$\frak{X}_{1}$ $=$ $\varepsilon _{0}I_{2},$ $\frak{X}_{2}$ $=$ $\left(
\varepsilon _{0}^{2}-2\right) I_{2},$
\end{center}

and

\begin{center}
$\frak{X}_{n}$ $=$ $\varepsilon _{0}\frak{X}_{n-1}-\frak{X}_{n-2},$
\end{center}

for all $n$ $\geq $ $3$ in $\Bbb{N}.$

Equivalently, the relation (3.2.7) holds.
\end{proof}

By the recurrence relation (3.2.7), we obtain that:

$\qquad \qquad X_{1}X_{2}+X_{2}X_{1}$ $=$ $\varepsilon _{0}I_{2},$

$\qquad \qquad \left( X_{1}X_{2}\right) ^{2}+\left( X_{2}X_{1}\right) ^{2}$ $%
=$ $\left( \varepsilon _{0}^{2}-2\right) I_{2},$

$\qquad \qquad \left( X_{1}X_{2}\right) ^{3}+(X_{2}X_{1})^{3}$ $=$ $\left(
\varepsilon _{0}^{3}-3\varepsilon _{0}\right) I_{2}.$

and

\begin{center}
$
\begin{array}{ll}
(X_{1}X_{2})^{4}+(X_{2}X_{1})^{4} & =\varepsilon _{0}\left( \varepsilon
_{0}^{3}-3\varepsilon _{0}\right) I_{2}-\left( \varepsilon _{0}^{2}-2\right)
I_{2} \\ 
& =\left( \varepsilon _{0}^{4}-4\varepsilon _{0}^{2}+2\right) I_{2},
\end{array}
$
\end{center}

in $M_{2}(\Bbb{C}).$

Inductively, one can verify that:

\begin{corollary}
There exists a functional sequence $\left( f_{n}\right) _{n=1}^{\infty },$
such that

(3.2.11)

\begin{center}
$\left( X_{1}X_{2}\right) ^{n}+(X_{2}X_{1})^{n}$ $=$ $\left(
f_{n}(\varepsilon _{0})\right) I_{2},$
\end{center}

in $M_{2}(\Bbb{C}).$ $\square $
\end{corollary}

\strut For instance, if $(f_{n})_{n=1}^{\infty }$ is in the sense of
(3.2.11), then

\begin{center}
$f_{1}(z)$ $=$ $z,$ $f_{2}(z)$ $=$ $z^{2}-2,$ $f_{3}(z)$ $=$ $z^{3}-3z,$
\end{center}

and

\begin{center}
$f_{4}(z)$ $=$ $z^{4}-4z^{2}+2,$
\end{center}

etc.

\strut The following theorem provides the refined result of (3.2.11).

\begin{theorem}
Let $X_{1}$ and $X_{2}$ be the generating matrices of a $2$-dimensional,
order-2, 2-generator subgroup $\frak{M}_{2}^{2}(2)$ of $M_{2}(\Bbb{C})$ in
the sense of (3.2.1), and let

\begin{center}
$\frak{X}_{l}$ $=$ $\left( X_{1}X_{2}\right) ^{l}+(X_{2}X_{1})^{l}$ $\in $ $%
M_{2}(\Bbb{C}),$
\end{center}

for all $l$ $\in $ $\Bbb{N}.$ Then there exists a functional sequence $%
\left( f_{n}\right) _{n=1}^{\infty }$ such that

\begin{center}
$\frak{X}_{n}$ $=$ $f_{n}(\varepsilon _{0})$ $I_{2}$ in $M_{2}(\Bbb{C}),$
\end{center}

where

\begin{center}
$\varepsilon _{0}$ $=$ $\det (X_{1})+\det (X_{2})-\det (X_{1}+X_{2}).$
\end{center}

Moreover,

(3.2.12)

\begin{center}
$
\begin{array}{ll}
f_{n}(z) & =\sum_{k=0}^{[\frac{n}{2}]}(-1)^{k}\frac{n}{k}\left( 
\begin{array}{c}
n-k-1 \\ 
k-1
\end{array}
\right) z^{n-2k} \\ 
&  \\ 
& =\sum_{k=0}^{[\frac{n}{2}]}(-1)^{k}\frac{n}{n-k}\left( 
\begin{array}{c}
n-k \\ 
k
\end{array}
\right) z^{n-2k},
\end{array}
$
\end{center}

where

\begin{center}
$[\frac{n}{2}]$ $=$ the maximal integer $\leq $ $\frac{n}{2},$
\end{center}

and

\begin{center}
$\left( 
\begin{array}{c}
m \\ 
l
\end{array}
\right) $ $=$ $\frac{m!}{l!(m-l)!},$ for all $m,$ $l$ $\in $ $\Bbb{N},$
\end{center}

for all $n$ $\in $ $\Bbb{N}.$
\end{theorem}

\begin{proof}
\strut By (3.2.7) and (3.2.11), there exists a functional sequence $%
(f_{n})_{n=1}^{\infty }$ such that

\begin{center}
$\frak{X}_{n}$ $=$ $f_{n}(\varepsilon _{0})$ $I_{2}$ in $M_{2}(\Bbb{C}),$
for all $n$ $\in $ $\Bbb{N},$
\end{center}

where

\begin{center}
$\varepsilon _{0}$ $=$ $\det (X_{1})+\det (X_{2})-\det (X_{1}+X_{2}).$
\end{center}

For instance,

(3.2.13)

\begin{center}
$f_{1}(z)$ $=$ $z,$ $f_{2}(z)$ $=$ $z^{2}-2,$ and $f_{3}(z)$ $=$ $z^{3}-3z,$
\end{center}

etc. So, it suffices to show that each $n$-th entry of the sequence $%
f_{n}(z) $ satisfies (3.2.12), for all $n$ $\in $ $\Bbb{N}.$

Say $n$ $=$ $1.$ Then the function $f_{1}(z)$ of (3.2.12) satisfies

$\qquad \qquad f_{1}(z)$ $=$ $\sum_{k=0}^{[\frac{1}{2}]}(-1)^{k}\frac{n}{n-k}%
\left( 
\begin{array}{c}
n-k-1 \\ 
k-1
\end{array}
\right) z^{n-2k}$

\strut

$\qquad \qquad \qquad =$ $\sum_{k=0}^{0}(-1)^{k}\frac{n}{n-k}\left( 
\begin{array}{c}
n-k \\ 
k
\end{array}
\right) z^{n-2k}$

\strut

$\qquad \qquad \qquad =$ $(-1)^{0}$ $\frac{1}{1-0}\left( 
\begin{array}{c}
1-0 \\ 
0
\end{array}
\right) z^{1-0}$ $=$ $z,$

satisfying

(3.2.14)

\begin{center}
$f_{1}(z)$ $=$ $z.$
\end{center}

If $n$ $=$ $2,$ then the function $f_{2}(z)$ of (3.2.12) goes to

$\qquad f_{2}(z)$ $=$ $\sum_{k=0}^{1}(-1)^{k}\frac{n}{n-k}\left( 
\begin{array}{c}
n-k \\ 
k
\end{array}
\right) z^{n-2k}$

\strut

$\qquad \qquad =$ $(-1)^{0}\frac{2}{2-0}\left( 
\begin{array}{c}
2-0 \\ 
0
\end{array}
\right) z^{2-0}$ $+$ $(-1)^{1}\frac{2}{2-1}\left( 
\begin{array}{c}
2-1 \\ 
1
\end{array}
\right) z^{2-2}$

\strut

$\qquad \qquad =$ $z^{2}$ $-1,$

satisfying

(3.2.15)

\begin{center}
$f_{2}(z)$ $=$ $z^{2}-1.$
\end{center}

\strut Therefore, the formula (3.2.12) holds true where $n$ $=$ $1,$ $2,$ by
(3.2.13).

\strut

Now, we consider the functions $f_{n}(z)$ of (3.2.12) satisfy the recurrence
relation;

(3.2.16)

\begin{center}
$f_{1}(z)$ $=$ $z,$ $f_{2}(z)$ $=$ $z^{2}-1,$
\end{center}

and

\begin{center}
$f_{n+1}(z)$ $=$ $zf_{n}(z)$ $-$ $f_{n-1}(z),$
\end{center}

for all $n$ $\geq $ $2$ in $\Bbb{N}.$

First of all, the functions $f_{1}$ and $f_{2}$ satisfy the initial
condition of the relation (3.2.16), by (3.2.14) and (3.2.15). So,
concentrate on the cases where $n$ $\geq $ $2$ in $\Bbb{N}.$

Observe that

\strut
\strut

$zf_{n}(z)-f_{n-1}(z)$

\strut

$=$ $z\left( \sum_{k=0}^{[\frac{n}{2}]}(-1)^{k}\text{ }\frac{n}{k}\left( 
\begin{array}{c}
n-k-1 \\ 
k-1
\end{array}
\right) z^{n-2k}\right) -\sum_{k=0}^{[\frac{n-1}{2}]}(-1)^{k}\frac{n-1}{k}%
\left( 
\begin{array}{c}
n-k-2 \\ 
k-1
\end{array}
\right) z^{n-1-2k}$

\strut

$=$ $\sum_{k=0}^{[\frac{n}{2}]}(-1)^{k}\frac{n}{k}\left( 
\begin{array}{c}
n-k-1 \\ 
k-1
\end{array}
\right) z^{n-2k+1}$ $-$ $\sum_{k=1}^{[\frac{n+1}{2}]}(-1)^{k-1}$ $\frac{n-1}{%
k-1}\left( 
\begin{array}{c}
n-k-1 \\ 
k-2
\end{array}
\right) z^{n-1-2k+2}$

\strut

$=$ $\sum_{k=0}^{[\frac{n}{2}]}(-1)^{k}\frac{n}{k}\left( 
\begin{array}{c}
n-k-1 \\ 
k-1
\end{array}
\right) z^{n+1-2k}+\sum_{k=1}^{[\frac{n+1}{2}]}(-1)^{k}\frac{n-1}{k-1}\left( 
\begin{array}{c}
n-k-1 \\ 
k-2
\end{array}
\right) z^{n+1-2k}$

\strut

$=$ $z^{n+1}+\sum_{k=1}^{[\frac{n}{2}]}(-1)^{k}\frac{n}{k}\left( 
\begin{array}{c}
n-k-1 \\ 
k-1
\end{array}
\right) z^{n+1-2k}+\sum_{k=1}^{[\frac{n+1}{2}]}(-1)^{k}\frac{n-1}{k-1}\left( 
\begin{array}{c}
n-k-1 \\ 
k-2
\end{array}
\right) z^{n+1-2k}$

\strut

$=$ $z^{n+1}+\sum_{k=1}^{[\frac{n}{2}]}(-1)^{k}\left( \frac{n}{k}\left( 
\begin{array}{c}
n-k-1 \\ 
k-1
\end{array}
\right) +\frac{n-1}{k-1}\left( 
\begin{array}{c}
n-k-1 \\ 
k-2
\end{array}
\right) \right) z^{n+1-2k}$

\strut

$+$ $(-1)^{[\frac{n+1}{2}]}\frac{n-1}{[\frac{n+1}{2}]-1}$ $\left( 
\begin{array}{c}
n-[\frac{n+1}{2}]-1 \\ 
\lbrack \frac{n+1}{2}]-1
\end{array}
\right) $ $z^{n+1-2[\frac{n+1}{2}]}$

\strut

$=$ $z^{n+1}+$ $\sum_{k=1}^{[\frac{n}{2}]}(-1)^{k}$ $\left( \frac{n(n-k-1)!}{%
k(n-k-1-k+1)!(k-1)!}+\frac{(n-1)(n-k-1)!}{(k-1)(n-k+k+2)!(k-2)!}\right)
z^{n+1-2k}$ 
\strut
$+ \enspace \frak{Y}$

\strut

where

(3.2.17)

$\hspace{35pt}\hspace{35pt}\frak{Y}$ $=$ $(-1)^{[\frac{n+1}{2}]}\frac{n-1}{[\frac{n+1}{2}]-1}$ $\left( 
\begin{array}{c}
n-[\frac{n+1}{2}]-1 \\ 
\lbrack \frac{n+1}{2}]-1
\end{array}
\right) $ $z^{n+1-2[\frac{n+1}{2}]},$

and then

\strut

$=$ $z^{n+1}+\sum_{k=1}^{[\frac{n}{2}]}\left( -1\right) ^{k}\left( \frac{%
n(n-k-1)!}{(n-2k)!k!}+\frac{n(n-k-1)!-(n-k-1)!}{(n-2k+1)!(k-1)!}\right)
z^{n+1-2k}$ $+$ $\frak{Y}$

\strut

$=$ $z^{n+1}+\sum_{k=1}^{[\frac{n}{2}]}(-1)^{k}\left( \frac{%
(n-2k+1)n(n-k-1)!+kn(n-k-1)!-k(n-k-1)!}{k(n-2k+1)!(k-1)!}\right) z^{n+1-2k}$ 
$+$ $\frak{Y}$

\strut

$=$ $z^{n+1}+\sum_{k=1}^{[\frac{n}{2}]}(-1)^{k}\left( \frac{%
(n^{2}-2kn+n+nk-k)(n-k-1)!}{k(n-2k+1)!(k-1)!}\right) z^{n+1-2k}+\frak{Y}$

\strut

$=$ $z^{n+1}+\sum_{k=1}^{[\frac{n}{2}]}(-1)^{k}\left( \frac{%
(n^{2}-nk+n-k)(n-k-1)!}{k(n-2k+1)!(k-1)!}\right) z^{n+1-2k}+\frak{Y}$

\strut

$=$ $z^{n+1}+\sum_{k=1}^{[\frac{n}{2}]}(-1)^{k}\left( \frac{%
(n+1)(n-k)(n-k-1)!}{k(n-2k+1)!(k-1)!}\right) z^{n+1-2k}+$ $\frak{Y}$

\strut

$=$ $z^{n+1}+\sum_{k=1}^{[\frac{n}{2}]}(-1)^{k}$ $\frac{n+1}{k}$ $\frac{%
(n-k)!}{(n-k-(k-1))!(k-1)!}$ $z^{n+1-2k}$ $+$ $\frak{Y}$

\strut

$=$ $z^{n+1}+\sum_{k=1}^{[\frac{n}{2}]}(-1)^{k}$ $\frac{n+1}{k}$ $\left( 
\begin{array}{c}
n-k \\ 
k-1
\end{array}
\right) $ $z^{n+1-2k}+$ $\frak{Y}$

\strut

$=$ $\sum_{k=0}^{[\frac{n}{2}]}(-1)^{k}$ $\frac{n+1}{k}$ $\left( 
\begin{array}{c}
n-k \\ 
k-1
\end{array}
\right) $ $z^{n+1-2k}+\frak{Y},$

\strut

i.e., for any $n$ $\geq $ $2,$ we obtain

(3.2.18)

\begin{center}
$zf_{n}(z)-f_{n-1}(z)$ $=$ $\sum_{k=0}^{[\frac{n}{2}]}(-1)^{k}$ $\frac{n+1}{k%
}$ $\left( 
\begin{array}{c}
n-k \\ 
k-1
\end{array}
\right) $ $z^{n+1-2k}+\frak{Y},$
\end{center}

where $\frak{Y}$ is in the sense of (3.2.17).

\strut

Notice that if $n$ $\geq $ $2$ is even in $\Bbb{N},$ then

\begin{center}
$[\frac{n}{2}]=[\frac{n+1}{2}]$ in $\Bbb{N},$
\end{center}

and hence, the first term of (3.2.18) contains the second one. Therefore,
one has

(3.2.19)

\begin{center}
$n$ $\in $ $2\Bbb{N}$ $\Rightarrow $ $zf_{n}(z)-f_{n-1}(z)$ $=$ $%
\sum_{k=0}^{[\frac{n+1}{2}]}(-1)^{k}$ $\frac{n+1}{k}$ $\left( 
\begin{array}{c}
n-k \\ 
k-1
\end{array}
\right) $ $z^{n+1-2k},$
\end{center}

if and only if

\begin{center}
$n$ $\in $ $2\Bbb{N}$ $\Longrightarrow $ $zf_{n}(z)-f_{n-1}(z)$ $=$ $%
f_{n+1}(z).$
\end{center}

\strut

Now, assume $n$ $\geq $ $2$ is odd in $\Bbb{N}.$ Then

$\qquad \frak{Y}$ $=$ $(-1)^{[\frac{n+1}{2}]}$ $\left( \frac{2n-2}{n+1-2}%
\right) $ $\left( \frac{\left( n-[\frac{n+1}{2}]-1\right) !}{\left( n-[\frac{%
n+1}{2}]-1-[\frac{n+1}{2}]+2\right) !\left( [\frac{n+1}{2}]-2\right) !}%
\right) $ $z^{n+1-2[\frac{n+1}{2}]} $

\strut

$\qquad \qquad =$ $(-1)^{[\frac{n+1}{2}]}\left( \frac{2n-2}{n-1}\right)
\left( \frac{\left( n-[\frac{n+1}{2}]-1\right) !}{\left( n-2[\frac{n+1}{2}%
]+1\right) !\left( [\frac{n+1}{2}]-2\right) !}\right) $ $z^{0}.$

\strut

$\qquad \qquad =$ $(-1)^{[\frac{n+1}{2}]}\left( 2\right) $ $\left( \frac{%
\left( n-[\frac{n+1}{2}]\right) !}{\left( n-2[\frac{n+1}{2}]+1\right)
!\left( n-[\frac{n+1}{2}]\right) \left( [\frac{n+1}{2}]-2\right) !}\right) $ $z^{0}.$

\strut

$\qquad \qquad =$ $(-1)^{[\frac{n+1}{2}]}$ $\left( \frac{n+1}{[\frac{n+1}{2}]%
}\right) \left( \frac{[\frac{n+1}{2}]-1}{n-[\frac{n+1}{2}]}\right) \left( 
\begin{array}{c}
n-[\frac{n+1}{2}] \\ 
\\ 
\lbrack \frac{n+1}{2}]-1
\end{array}
\right) $ $z^{0}.$

\strut

$\qquad \qquad =$ $(-1)^{[\frac{n+1}{2}]}\left( \frac{n+1}{[\frac{n+1}{2}]}%
\right) \left( \frac{n+1-2}{2n-n-1}\right) \left( 
\begin{array}{c}
n-[\frac{n+1}{2}] \\ 
\\ 
\lbrack \frac{n+1}{2}]-1
\end{array}
\right) $ $z^{0}.$

\strut

$\qquad \qquad =$ $\left( (-1)^{[\frac{n+1}{2}]}\left( \frac{n+1}{[\frac{n+1%
}{2}]}\right) \left( 
\begin{array}{c}
n-[\frac{n+1}{2}] \\ 
\\ 
\lbrack \frac{n+1}{2}]-1
\end{array}
\right) \right) $ $z^{0}.$

i.e.,

(3.2.20)

\begin{center}
$n$ $\in $ $2\Bbb{N}+1$ $\Longrightarrow $ $\frak{Y}$ $=$ $\left( (-1)^{[%
\frac{n+1}{2}]}\left( \frac{n+1}{[\frac{n+1}{2}]}\right) \left( 
\begin{array}{c}
n-[\frac{n+1}{2}] \\ 
\\ 
\lbrack \frac{n+1}{2}]-1
\end{array}
\right) \right) $ $z^{0}.$
\end{center}

\strut

So, by (3.2.20), if $n$ $\geq $ $2$ is odd in $\Bbb{N},$ then

(3.2.21)

$\hspace{35pt}\hspace{35pt}zf_{n}(z)-f_{n-1}(z)$ $=$ $\sum_{k=0}^{[\frac{n+1}{2}]}(-1)^{k}$ $\frac{n+1%
}{k}$ $\left( 
\begin{array}{c}
n-k \\ 
k-1
\end{array}
\right) $ $z^{n+1-2k},$

if and only if

\begin{center}
$zf_{n}(z)-f_{n-1}(z)$ $=$ $f_{n+1}(z).$
\end{center}

\strut

Therefore, by (3.2.19) and (3.2.21), one can conclude that

(3.2.22)

\begin{center}
$n$ $\geq $ $2$ in $\Bbb{N}$ $\Longrightarrow $ $zf_{n}(z)-f_{n-1}(z)$ $=$ $%
f_{n+1}(z).$
\end{center}

\strut

The above result (3.2.22) shows that

\begin{center}
$f_{n+1}(\varepsilon _{0})$ $=$ $\varepsilon _{0}f_{n}(\varepsilon
_{0})-f_{n-1}(\varepsilon _{0}),$ for $n$ $\geq $ $2$ in $\Bbb{N},$
\end{center}

with

\begin{center}
$f_{1}(\varepsilon _{0})$ $=$ $\varepsilon _{0},$ and $f_{2}(\varepsilon
_{0})$ $=$ $\varepsilon _{0}^{2}-2.$
\end{center}

Recall that, by (3.2.7),

\begin{center}
$\frak{X}_{1}$ $=$ $\varepsilon _{0}$ $I_{2},$ and $\frak{X}_{2}$ $=$ $%
\left( \varepsilon _{0}^{2}-2\right) $ $I_{2},$
\end{center}

and

\begin{center}
$\frak{X}_{n+1}$ $=$ $\varepsilon _{0}\frak{X}_{n}$ $-$ $\frak{X}_{n-1}$ $=$ 
$f_{n}(\varepsilon _{0})$ $I_{2}.$
\end{center}

Therefore, one can conclude that if

\begin{center}
$f_{n}(z)$ $=$ $\sum_{k=0}^{[\frac{n}{2}]}(-1)^{k}\left( \frac{n}{k}\right)
\left( 
\begin{array}{c}
n-k-1 \\ 
k-1
\end{array}
\right) z^{n-2k},$
\end{center}

then

\begin{center}
$\frak{X}_{n}$ $=$ $f_{n}(\varepsilon _{0})$ $I_{2},$ for all $n$ $\in $ $%
\Bbb{N}.$
\end{center}

\strut
\end{proof}

\strut By the recurrence relation (3.2.7), we obtain the above main theorem
of Section 3.

\strut \strut \strut

\textbf{Summary of Section 3.2} If $X_{1}$ and $X_{2}$ are the generating
matrices of the 2-dimensional, order-2, 2-generator subgroup $\frak{M}%
_{2}^{2}(2)$ of $M_{2}(\Bbb{C}),$ then

(3.2.23)

\begin{center}
$(X_{1}X_{2})^{n}+(X_{2}X_{1})^{n}$ $=$ $f_{n}(\varepsilon _{0})$ $I_{2},$
\end{center}

in $M_{2}(\Bbb{C}),$ where

\begin{center}
$f_{n}(z)$ $=$ $\sum_{k=0}^{[\frac{n}{2}]}(-1)^{k}\left( \frac{n}{k}\right)
\left( 
\begin{array}{c}
n-k-1 \\ 
k-1
\end{array}
\right) z^{n-2k},$
\end{center}

and

\begin{center}
$\varepsilon _{0}$ $=$ $\det (X_{1})+\det (X_{2})-\det (X_{1}+X_{2}),$
\end{center}

for all $n$ $\in $ $\Bbb{N}.$ $\blacksquare $

\strut

\strut Let $x_{1}$ and $x_{2}$ be the generators of the group $\Gamma
_{2}^{2}$ of (1.0.1), and let

\begin{center}
$\alpha ^{2}(x_{1})$ $=$ $X_{1}$ and $\alpha ^{2}(x_{2})$ $=$ $X_{2}$
\end{center}

be the corresponding generating matrices of the isomorphic group $\frak{M}%
_{2}^{2}(2)$ of $\Gamma _{2}^{2}$ in $M_{2}(\Bbb{C}).$ The main result
(3.2.23) shows that the generators $x_{1}$ and $x_{2}$ of $\Gamma _{2}^{2}$
induces the analytic data in $M_{2}(\Bbb{C})$ depending on the elements $%
x_{1}x_{2}$ and its inverses $x_{2}x_{1}$ up to (3.2.23).

\subsection{\strut Group Algebras $A_{\alpha ^{n},\text{ }\Gamma _{N}^{2}}$}

Let $\Gamma $ be an arbitrary discrete group, and let $(H,$ $\alpha )$ be a 
\emph{Hilbert-space representation of} $\Gamma $ consisting of a Hilbert
space $H$ and the group-action

\begin{center}
$\alpha $ $:$ $\Gamma $ $\rightarrow $ $L(H),$
\end{center}

making

\begin{center}
$\alpha (g)$ $:$ $H$ $\rightarrow $ $H,$ for all $g$ $\in $ $\Gamma $
\end{center}

be linear operators (or linear transformations) on $H,$ satisfying

\begin{center}
$\alpha (g_{1}g_{2})$ $=$ $\alpha (g_{1})$ $\alpha (g_{2}),$ for all $g_{1},$
$g_{2}$ $\in $ $\Gamma ,$
\end{center}

and

\begin{center}
$\alpha (g^{-1})$ $=$ $\left( \alpha (g)\right) ^{-1},$ for all $g$ $\in $ $%
\Gamma ,$
\end{center}

where $L(H)$ means the operator algebra consisting of all linear operators
on $H.$

Remark that, if we are working with \emph{topologies}, then one may replace $%
L(H)$ to $B(H),$ the operator algebra consisting of all \emph{bounded} (or 
\emph{continuous}) linear operators on $H.$ However, we are not considering
topologies throughout this paper, so we simply set the representational
models for $L(H)$ here.

For example, our group $\Gamma _{N}^{2}$ of (1.0.1) has its Hilbert-space
representations $\left( \Bbb{C}^{n},\text{ }\alpha ^{n}\right) ,$ for all $n$
$\in $ $\Bbb{N}$ $\setminus $ $\{1\},$ and each element $g$ of $\Gamma
_{N}^{2}$ is realized as $\alpha ^{n}(g)$ $\in $ $\frak{M}_{n}^{2}(N)$ in

(3.3.1)

\begin{center}
$M_{n}(\Bbb{C})$ $=$ $L(\Bbb{C}^{n})$ $=$ $B(\Bbb{C}^{n}).$
\end{center}

The first equivalence relation ($=$) of (3.3.1) holds because, under
finite-dimensionality, all linear operators on $\Bbb{C}^{n}$ are matrices
acting on $\Bbb{C}^{n},$ and vice versa. The second equivalence relation ($=$%
) of (3.3.1) holds, because, under finite-dimensionality, all pure-algebraic
operators on $\Bbb{C}^{n}$ are automatically bounded (and hence, continuous).

For an arbitrary group $\Gamma ,$ realized by $(H,$ $\alpha ),$ one obtains
the isomorphic group $\alpha (\Gamma )$ in $L(H),$ i.e.,

\begin{center}
$\Gamma $ $\overset{\text{Group}}{=}$ $\alpha (\Gamma )$ in $L(H).$
\end{center}

So, one can construct a subalgebra,

(3.3.2)

\begin{center}
$A_{\alpha ,\,\Gamma }$ $=$ $\Bbb{C}\left[ \alpha (\Gamma )\right] $ of $%
L(H),$
\end{center}

\strut where $\Bbb{C}[X]$ mean the \emph{polynomial ring} in sets $X.$ Such
rings $\Bbb{C}[X]$ form algebras under polynomial addition and polynomial
multiplication over $\Bbb{C},$ and hence, under the inherited operator
addition and operator multiplication, $\Bbb{C}\left[ \alpha (\Gamma )\right] 
$ forms an algebra in $L(H)$ in (3.3.2).

\begin{definition}
Let $\Gamma $ be a group and let $(H,$ $\alpha )$ be a Hilbert-space
representation of $\Gamma .$ The subalgebra $A_{\alpha ,\,\Gamma }$ of $L(H)$
in the sense of (3.3.2) is called the group algebra of $\Gamma $ induced by $%
(H,$ $\alpha ).$
\end{definition}

It is not difficult to show that if

(3.3.3)

\begin{center}
$A_{\Gamma }^{o}$ $=$ $\Bbb{C}[\Gamma ],$
\end{center}

the pure-algebraic algebra generated by $\Gamma ,$ as a polynomial ring in $%
\Gamma ,$ then

(3.3.4)

\begin{center}
$A_{\Gamma }^{o}$ $\overset{\text{Alg}}{=}$ $A_{\alpha ,\text{ }\Gamma }$,
\end{center}

for all Hilbert-space representations $(H,$ $\alpha )$ of $\Gamma $
algebraically, where ``$\overset{\text{Alg}}{=}$'' means ``being
algebra-isomorphic.''

\begin{proposition}
Let $\Gamma _{N}^{2}$ be the group of (1.0.1), and let $\left( \Bbb{C}^{n},%
\text{ }\alpha ^{n}\right) $ be our $n$-dimensional representations of $%
\Gamma _{N}^{2},$ for all $n$ $\in $ $\Bbb{N}$ $\setminus $ $\{1\}.$ Let $%
\frak{M}_{n}^{2}(N)$ be the $n$-dimensional, order-2, $N$-generator
subgroups of $M_{n}(\Bbb{C}),$ for all $n$ $\in $ $\Bbb{N}$ $\setminus $ $%
\{1\}.$ If $A_{\alpha ^{n},\,\Gamma _{N}^{2}}$ are in the sense of (3.3.2),
and $A_{\Gamma _{N}^{2}}$ is the group algebra in the sense of (3.3.3), then

(3.3.5)

\begin{center}
$A_{\Gamma _{N}^{2}}$ $\overset{\text{Alg}}{=}$ $A_{\alpha ^{n},\text{ }%
\Gamma _{N}^{2}}$ $=$ $\Bbb{C}\left[ \frak{M}_{n}^{2}(N)\right] $, for all $%
n $ $\in $ $\Bbb{N}.$
\end{center}
\end{proposition}

\begin{proof}
The proof of (3.3.5) is from the general relation (3.3.4). Recall that, by
the group-isomorphism $\alpha ^{n},$

$\alpha ^{n}\left( \Gamma _{2}^{2}\right) $ $=$ $\frak{M}_{n}^{2}(N),$ for
all $n$ $\in $ $\Bbb{N}.$

\strut
\end{proof}

\subsection{\strut A Group Algebra $A_{\alpha ^{2},\text{ }\Gamma _{2}^{2}}$}

In this section, we take the group algebra in the sense of (3.3.2),

(3.4.1)

\begin{center}
$A_{\alpha ^{2},\Gamma _{2}^{2}}$ $\overset{denote}{=}$ $\mathcal{A}_{2}$,
\end{center}

as a subalgebra of $M_{2}(\Bbb{C})$ $=$ $L(\Bbb{C}^{2})$, under our
representation $\left( \Bbb{C}^{2},\text{ }\alpha ^{2}\right) .$ \strut By
(3.3.5), pure-algebraically, the algebra $\mathcal{A}_{2}$ of (3.4.1) is
algebra-isomorphic to the group algebra $A_{\Gamma _{2}^{2}}$ $=$ $\Bbb{C}%
[\Gamma _{2}^{2}].$

By (3.1.5), all elements of $\mathcal{A}_{2}$ are the linear combinations of
the matrices formed by

(3.4.2)

\begin{center}
$I_{2},$ $X_{1},$ $X_{2},$ $(X_{1}X_{2})^{n},$ $(X_{2}X_{1})^{n},$
\end{center}

or

\begin{center}
$\left( X_{1}X_{2}\right) ^{n}X_{1},$ $\left( X_{2}X_{1}\right) ^{n}X_{2},$
\end{center}

for all $n$ $\in $ $\Bbb{N},$ where $X_{j}$ $=$ $\alpha ^{2}(x_{j})$, for $j$
$=$ $1,$ $2,$ since

\begin{center}
$\mathcal{A}_{2}$ $=$ $\Bbb{C}\left[ \frak{M}_{2}^{2}(2)\right] $ in $M_{2}(%
\Bbb{C}).$
\end{center}

As we have seen in Section 3.1, the building blocks of $\mathcal{A}_{2}$ in
the sense of (3.4.2) satisfy that: (i) $I_{2},$ $X_{1},$ $X_{2},$ $%
(X_{1}X_{2})^{n}X_{1},$ and $(X_{2}X_{1})^{n}X_{2}$ are self-invertible in $%
\mathcal{A}_{2}$ $\subset $ $M_{2}(\Bbb{C}),$ and (ii)

\begin{center}
$\left( (X_{1}X_{2})^{n}\right) ^{-1}$ $=$ $\left( X_{2}X_{1}\right) ^{n},$
\end{center}

for all $n$ $\in $ $\Bbb{N}.$

\begin{corollary}
Let $X_{1}$ and $X_{2}$ be the generating matrices of $\frak{M}_{2}^{2}(2),$
and hence, those of $\mathcal{A}_{2}.$ Then

(3.4.3)

\begin{center}
$\left( (X_{1}X_{2})^{n}\right) ^{-1}$ $=$ $f_{n}(\varepsilon _{0})$ $%
I_{2}-(X_{1}X_{2})^{n},$
\end{center}

and

\begin{center}
$\left( (X_{2}X_{1})\right) ^{-1}$ $=$ $f_{n}(\varepsilon _{0})$ $I_{2}$ $-$ 
$\left( X_{2}X_{1}\right) ^{n}$
\end{center}

for all $n$ $\in $ $\Bbb{N},$ where $(f_{n})_{n=2}^{\infty }$ is the
functional sequence in the sense of (3.2.12).
\end{corollary}

\begin{proof}
\strut By (3.2.23), we have

\begin{center}
$(X_{1}X_{2})^{n}+(X_{2}X_{1})^{n}$ $=$ $f_{n}(\varepsilon _{0})$ $I_{2}$ in 
$\mathcal{A}_{2},$
\end{center}

where $f_{n}$ are in the sense of (3.2.12), and $\varepsilon _{0}$ is in the
sense of (3.2.5), for all $n$ $\in $ $\Bbb{N}$. So,

\begin{center}
$(X_{2}X_{1})^{n}$ $=$ $f_{n}(\varepsilon _{0})I_{2}$ $-$ $\left(
X_{1}X_{2}\right) ^{n}$ in $\mathcal{A}_{2},$
\end{center}

for all $n$ $\in $ $\Bbb{N}.$ Thus,

\begin{center}
$\left( (X_{1}X_{2})^{n}\right) ^{-1}$ $=$ $f_{n}(\varepsilon _{0})$ $I_{2}$ 
$-$ $\left( X_{1}X_{2}\right) ^{n},$
\end{center}

for all $n$ $\in $ $\Bbb{N}.$

Similarly, we obtain

\begin{center}
$\left( (X_{2}X_{1})^{n}\right) ^{-1}$ $=$ $f_{n}(\varepsilon _{0})$ $I_{2}$ 
$-$ $(X_{2}X_{1})^{n},$
\end{center}

for all $n$ $\in $ $\Bbb{N}.$
\end{proof}

\strut It means that if $T$ is a matrix of $\mathcal{A}_{2}$ in the form of
either $(X_{1}X_{2})^{n}$ or $(X_{2}X_{1})^{n},$ for any $n$ $\in $ $\Bbb{N}%
, $ then

\begin{center}
$T^{-1}$ $=$ $f_{n}(\varepsilon _{0})$ $I_{2}$ $-T.$
\end{center}

\begin{proposition}
Let $A$ $\in $ $\frak{M}_{2}^{2}(2)$ in $\mathcal{A}_{2}.$ Then

(3.4.4)

\begin{center}
$A^{-1}$ $=$ $\left\{ 
\begin{array}{lll}
A &  & \text{or} \\ 
r_{A}\text{ }I_{2}-A, &  & 
\end{array}
\right. $
\end{center}

for some $r_{A}$ $\in $ $\Bbb{C}.$
\end{proposition}

\begin{proof}
\strut If $A$ $\in $ $\frak{M}_{2}^{2}(2)$ in $\mathcal{A}_{2},$ then $T$ is
one of the forms of (3.4.2). As we discussed above, if $A$ is $I_{2},$ or $%
X_{1},$ or $X_{2},$ or $(X_{1}X_{2})^{n}X_{1},$ or $(X_{2}X_{1})^{n}X_{2},$
for $n$ $\in $ $\Bbb{N},$ then it is self-invertible in $\mathcal{A}_{2}.$
i.e.,

\begin{center}
$A^{-1}$ $=$ $A$ in $\mathcal{A}_{2}.$
\end{center}

If $A$ is either $\left( X_{1}X_{2}\right) ^{n},$ or $(X_{2}X_{1})^{n},$ for
any $n$ $\in $ $\Bbb{N},$ then, by (3.4.3), there exists $r_{A}$ $\in $ $%
\Bbb{C},$ such that

\begin{center}
$A^{-1}$ $=$ $r_{A}-A$ in $\mathcal{A}_{2}\vdots .$
\end{center}

\strut \strut
\end{proof}

The above proposition characterizes the invertibility of $\frak{M}%
_{2}^{2}(2) $ ``in $\mathcal{A}_{2}.$''

\subsection{\strut Trace of Certain Matrices of $\mathcal{A}_{2}$}

\strut Throughout this section, we will use the same notations used before.
Let $M_{n}(\Bbb{C})$ be the $(n\times n)$-matricial algebra, for $n$ $\in $ $%
\Bbb{N}.$ Then the \emph{trace }$tr$\emph{\ on} $M_{n}(\Bbb{C})$ is
well-defined as a linear functional on $M_{n}(\Bbb{C})$ by

\begin{center}
$tr\left( 
\begin{array}{cccc}
a_{11} & a_{12} & \cdots & a_{1n} \\ 
a_{21} & a_{22} & \cdots & a_{2n} \\ 
\vdots & \vdots & \ddots & \vdots \\ 
a_{n1} & a_{n2} & \cdots & a_{nn}
\end{array}
\right) $ $=$ $\sum_{j=1}^{n}$ $a_{jj}.$
\end{center}

For instance, if $n$ $=$ $2,$ then

\begin{center}
$tr\left( 
\begin{array}{cc}
a & b \\ 
c & d
\end{array}
\right) $ $=$ $a+d.$
\end{center}

Now, let $\mathcal{A}_{2}$ be the group algebra of $\Gamma _{2}^{2}$ induced
by $\left( \Bbb{C}^{2},\text{ }\alpha ^{2}\right) $ in $M_{2}(\Bbb{C}),$ in
the sense of (3.4.1)$.$ Since the algebra $\mathcal{A}_{2}$ is a subalgebra
of $M_{2}(\Bbb{C}),$ one can naturally restrict the trace $tr$ on $M_{2}(%
\Bbb{C})$ to that on $\mathcal{A}_{2}.$ \strut i.e., the pair $\left( 
\mathcal{A}_{2},\text{ }tr\right) $ forms a \emph{noncommutative free
probability space} in the sense of [8] and [9].

In this section, we are interested in trace of certain types of matrices in $%
\mathcal{A}_{2}.$ Let $X_{1}$ and $X_{2}$ be the generating matrices of $%
\frak{M}_{2}^{2}(2),$ and hence, those of $\mathcal{A}_{2}.$ Define a new
element $T$ of $\mathcal{A}_{2}$ by

(3.5.1)

\begin{center}
$T$ $=$ $X_{1}+X_{2}$ $\in $ $\mathcal{A}_{2}.$
\end{center}

By the self-invertibility of $X_{1}$ and $X_{2},$ this matrix is understood
as the \emph{radial operator} (e.g., [5]) \emph{of} $\mathcal{A}_{2}.$

Observe that

$\qquad T^{2}$ $=$ $\left( X_{1}+X_{2}\right) ^{2}$ $=$ $X_{1}^{2}+\left(
X_{1}X_{2}+X_{2}X_{1}\right) +X_{2}^{2}$

$\qquad \qquad =$ $I_{2}+f_{1}(\varepsilon _{0})$ $I_{2}$ $+$ $I_{2}$ $=$ $%
\left( f_{1}(\varepsilon _{0})+2\right) $ $I_{2},$

and

$\qquad T^{3}$ $=$ $\left( X_{1}+X_{2}\right) ^{3}$ $=$ $\left(
X_{1}+X_{2}\right) ^{2}(X_{1}+X_{2})$

$\qquad \qquad =$ $\left( f_{1}(\varepsilon _{0})+2\right) T,$

and

\begin{center}
$T^{4}$ $=$ $T^{3}T$ $=$ $\left( f_{1}(\varepsilon _{0})+2\right) T^{2}$ $=$ 
$\left( f_{1}(\varepsilon _{0})+2\right) ^{2}I_{2},$
\end{center}

and

\begin{center}
$T^{5}$ $=$ $T^{4}T$ $=$ $\left( f_{1}(\varepsilon _{0})+2\right) ^{2}T,$
\end{center}

and

\begin{center}
$T^{6}$ $=$ $T^{5}T$ $=$ $\left( f_{1}(\varepsilon _{0})+2\right) ^{2}T^{2}$ 
$=$ $\left( f_{1}(\varepsilon _{0})+2\right) ^{3}I_{2},$
\end{center}

etc, where $f_{n}$ are in the sense of (3.2.12),

\begin{center}
$f_{n}(z)$ $=$ $\sum_{k=0}^{[\frac{n}{2}]}$ $(-1)^{k}\left( \frac{n}{k}%
\right) \left( 
\begin{array}{c}
n-k-1 \\ 
k-1
\end{array}
\right) z^{n-2k},$
\end{center}

for all $n$ $\in $ $\Bbb{N},$ and hence,

\begin{center}
$f_{1}(\varepsilon _{0})$ $=$ $\varepsilon _{0}.$
\end{center}

Inductively, one obtains that:

\begin{theorem}
Let $T$ $=$ $X_{1}+X_{2}$ be the radial operator of $\mathcal{A}_{2},$ where 
$X_{1}$ and $X_{2}$ are the generating matrices of $\mathcal{A}_{2}.$ Then

(3.5.2)

\begin{center}
$T^{2n}$ $=$ $\left( \varepsilon _{0}+2\right) ^{n}I_{2},$ and $T^{2n+1}$ $=$
$\left( \varepsilon _{0}+2\right) ^{n}T,$
\end{center}

in $\mathcal{A}_{2},$ for all $n$ $\in $ $\Bbb{N},$ where

\begin{center}
$\varepsilon _{0}$ $=$ $\det (X_{1})+\det (X_{2})-\det (X_{1}+X_{2}).$
\end{center}
\end{theorem}

\begin{proof}
The proof of (3.5.2) is done inductively by the observations in the very
above paragraphs. Indeed, one can get that:

\begin{center}
$T^{2n}$ $=$ $\left( f_{1}(\varepsilon _{0})+2\right) ^{n}I_{2},$
\end{center}

and

\begin{center}
$T^{2n+1}=\left( f_{1}(\varepsilon _{0})+2\right) ^{n}T,$
\end{center}

in $\mathcal{A}_{2},$ for all $n$ $\in $ $\Bbb{N}.$ And, since $f_{1}(z)$ $=$
$z,$ we obtain

\begin{center}
$f_{1}(\varepsilon _{0})$ $=$ $\varepsilon _{0},$
\end{center}

where $\varepsilon _{0}$ is in the sense of (3.2.5).
\end{proof}

By (3.5.2), we obtain the following data.

\begin{corollary}
Let $T$ $=$ $X_{1}+X_{2}$ be the radial operator of $\mathcal{A}_{2}.$ Then

(3.5.3)

\begin{center}
$tr\left( T^{n}\right) $ $=$ $\left\{ 
\begin{array}{ll}
2\left( \varepsilon _{0}+2\right) ^{n/2} & \text{if }n\text{ is even} \\ 
0 & \text{if }n\text{ is odd}
\end{array}
\right. $
\end{center}

for all $n$ $\in $ $\Bbb{N}.$
\end{corollary}

\begin{proof}
Let's set

(3.5.4)

\begin{center}
$X_{j}$ $=$ $\left( 
\begin{array}{cc}
t_{j} & s_{j} \\ 
\frac{1-t_{j}^{2}}{s_{j}} & -t_{j}
\end{array}
\right) $ in $\mathcal{A}_{2},$ for $j$ $=$ $1,$ $2,$
\end{center}

\strut

where $(t_{1},$ $s_{1})$ and $(t_{2},$ $s_{2})$ are strongly distinct pair
in $\Bbb{C}^{\times }\times \Bbb{C}^{\times }.$

By (3.5.2), if $T$ is given as above in $\mathcal{A}_{2},$ then

\begin{center}
$T^{2n}$ $=$ $\left( \varepsilon _{0}+2\right) ^{n}$ $I_{2},$ and $T^{2n+1}$ 
$=$ $\left( \varepsilon _{0}+2\right) ^{n}T,$
\end{center}

for all $n$ $\in $ $\Bbb{N}.$

Thus, it is not difficult to check that

$\qquad tr\left( T^{2n}\right) $ $=$ $tr\left( \left( 
\begin{array}{cc}
(\varepsilon _{0}+2)^{n} & 0 \\ 
0 & \left( \varepsilon _{0}+2\right) ^{n}
\end{array}
\right) \right) $

$\strut $

$\qquad \qquad \quad =$ $\left( \varepsilon _{0}+2\right) ^{n}+\left(
\varepsilon _{0}+2\right) ^{n}$ $=$ $2\left( \varepsilon _{0}+2\right) ^{n},$

for all $n$ $\in $ $\Bbb{N},$ i.e.,

(3.5.5)

\begin{center}
$tr(T^{2n})$ $=$ $2\left( \varepsilon _{0}+2\right) ^{n},$ for all $n$ $\in $
$\Bbb{N}.$
\end{center}

\strut

By the direct computation and by (3.5.4), one has

(3.5.6)

\begin{center}
$T$ $=$ $\left( 
\begin{array}{ccc}
t_{1}+t_{2} &  & s_{1}+s_{2} \\ 
&  &  \\ 
\frac{1-t_{1}^{2}}{s_{1}}+\frac{1-t_{2}^{2}}{s_{2}} &  & -(t_{1}+t_{2})
\end{array}
\right) $
\end{center}

in $\mathcal{A}_{2}.$ So, one can get that:

\begin{center}
$tr(T)$ $=$ $\left( t_{1}+t_{2}\right) +\left( -(t_{1}+t_{2})\right) $ $=$ $%
0,$
\end{center}

by (3.5.6).

Also, again by (3.5.2), we have

\begin{center}
$
\begin{array}{ll}
tr(T^{2n+1}) & =tr\left( (\varepsilon _{0}+2)^{n}T\right) \\ 
& =\left( \varepsilon _{0}+2\right) ^{n}tr(T)=0,
\end{array}
$
\end{center}

for all $n$ $\in $ $\Bbb{N}.$ i.e.,

\begin{center}
$tr(T)$ $=$ $0$ $=$ $tr\left( T^{2n+1}\right) ,$ for all $n$ $\in $ $\Bbb{N}%
. $
\end{center}

In summary, we have

(3.5.7)

\begin{center}
$tr\left( T^{2n-1}\right) $ $=$ $0,$ for all $n$ $\in $ $\Bbb{N}.$
\end{center}

\strut

Therefore, by (3.5.5) and (3.5.7), we obtain the formula (3.5.3).
\end{proof}

\section{Application: \strut $\mathcal{A}_{2}$ and Lucas Numbers}

In this section, we will use the same notations used in Section 3 above. For
instance, let

\begin{center}
$\mathcal{A}_{2}$ $=$ $\Bbb{C}\left[ \alpha ^{2}(\Gamma _{2}^{2})\right] $ $%
= $ $\Bbb{C}\left[ \frak{M}_{2}^{2}(2)\right] $
\end{center}

is the group algebra of $\Gamma _{2}^{2}$ induced by our 2-dimensional
representation $\left( \Bbb{C}^{2},\text{ }\alpha ^{2}\right) $ in $M_{2}(%
\Bbb{C}).$

Here, as application of Sections 2 and 3, we study connections between
analytic data obtained from elements of $\mathcal{A}_{2}$ and \emph{Lucas
numbers}.

For more details about Lucas numbers, see [7], [8], [9], [10], [11] and
cited papers therein.

From the main result (3.2.23) of Section 3, we find a functional sequence $%
(f_{n})_{n=1}^{\infty }$, with its $n$-th entries

(4.1)

\begin{center}
$f_{n}(z)$ $=$ $\sum_{k=0}^{[\frac{n}{2}]}(-1)^{k}\left( \frac{n}{k}\right)
\left( 
\begin{array}{c}
n-k-1 \\ 
k-1
\end{array}
\right) z^{n-2k}$.
\end{center}

for all $n$ $\in $ $\Bbb{N},$ satisfying

\begin{center}
$\left( X_{1}X_{2}\right) ^{n}+\left( (X_{1}X_{2})^{n}\right) ^{-1}$ $=$ $%
f_{n}(\varepsilon _{0})$ $I_{2}$ in $\mathcal{A}_{2},$
\end{center}

\strut where

\begin{center}
$\varepsilon _{0}$ $=$ $\det (X_{1})+\det (X_{2})-\det \left(
X_{1}+X_{2}\right) .$
\end{center}

For any arbitrarily fixed $n$ $\in $ $\Bbb{N},$ consider the
(non-alternative parts of) summands

(4.2)

\begin{center}
\strut $\left( \frac{n}{k}\right) \left( 
\begin{array}{c}
n-k-1 \\ 
k-1
\end{array}
\right) ,$ for $k$ $=$ $0,$ $1,$ ..., $[\frac{n}{2}],$
\end{center}

of $f_{n}(z)$ in (4.1).

Define such quantities of (4.2) by a form of a function,

\begin{center}
$g$ $:$ $\Bbb{N}\times \Bbb{N}_{0}$ $\rightarrow $ $\Bbb{C}$ by
\end{center}

(4.3)

\begin{center}
$g(n,$ $k)$ $\overset{def}{=}$ $\left\{ 
\begin{array}{ll}
\left( \frac{n}{k}\right) \left( 
\begin{array}{c}
n-k-1 \\ 
k-1
\end{array}
\right) & \text{if }k=0,1,\text{ ..., }[\frac{n}{2}] \\ 
&  \\ 
0 & \text{otherwise,}
\end{array}
\right. $
\end{center}

for all $\left( n,\text{ }k\right) $ $\in $ $\Bbb{N}\times \Bbb{N}_{0},$
where

\begin{center}
$\Bbb{N}_{0}$ $\overset{def}{=}$ $\Bbb{N}$ $\cup $ $\{0\}.$
\end{center}

\strut Note that the definition (4.3) of the function $g$ represents the
(non-alternating parts) of summands of $(f_{n})_{n=1}^{\infty }.$

The authors are not sure if the following theorem is already proven or not.
They could not find the proofs in their references. So, we provide the
following proof of the theorem.

\begin{theorem}
Let $g$ be a function from $\Bbb{N}\times \Bbb{N}_{0}$ into $\Bbb{C}$ be in
the sense of (4.3), Then

(4.4)

\begin{center}
$g(1,$ $0)$ $=$ $1,$
\end{center}

and

\begin{center}
$g\left( n+1,\text{ }k\right) $ $=$ $g(n,$ $k)$ $+$ $g(n-1,$ $k-1),$
\end{center}

for all $(n,$ $k)$ $\in $ $\left( \Bbb{N}\text{ }\setminus \text{ }1\right) $
$\times $ $\Bbb{N}$.
\end{theorem}

\begin{proof}
Observe first that:

(4.5)

\begin{center}
$
\begin{array}{ll}
g(n,k) & =\frac{n}{k}\left( 
\begin{array}{c}
n-k-1 \\ 
k-1
\end{array}
\right) =\frac{n}{k}\frac{\left( n-k-1\right) !}{(k-1)!(n-k-1-k+1)!} \\ 
&  \\ 
& =\frac{n(n-k-1)!}{(n-2k)!k!},
\end{array}
$
\end{center}

\strut whenever $g(n,$ $k)$ is non-zero, for $(n,$ $k)$ $\in $ $\Bbb{N}$ $%
\times $ $\Bbb{N}_{0}.$

\strut By (4.5), we obtain that

\strut (4.6)

\begin{center}
$g(1,$ $0)$ $=$ $\frac{1(1-0-1)!}{(1-0)!0!}$ $=$ $1,$
\end{center}

\strut

\begin{center}
$g(2,$ $0)$ $=$ $\frac{2(2-0-1)!}{(2-0)!0!}$ $=$ $1,$
\end{center}

and

\begin{center}
$g(2,$ $1)$ $=$ $\frac{2(2-1-1)!}{(2-2)!1!}$ $=$ $2.$
\end{center}

\strut

Consider now that:

$\qquad g(n,$ $k)+g(n-1,$ $k-1)$

\strut

$\qquad \qquad =$ $\frac{n}{k}\left( 
\begin{array}{c}
n-k-1 \\ 
k-1
\end{array}
\right) +\frac{n-1}{k-1}\left( 
\begin{array}{c}
n-1-(k-1)-1 \\ 
(k-1)-1
\end{array}
\right) $

\strut

$\qquad \qquad =$ $\frac{n(n-k-1)!}{(n-2k)!k!}$ $+$ $\frac{(n-1)(n-k-1)!}{%
(n-2k+1)!(k-1)!}$

by (4.5)

$\qquad \qquad =$ $\frac{n(n-2k+1)(n-k-1)!+k(n-1)(n-k-1)!}{(n-kk+1)!k!}$

\strut

$\qquad \qquad =$ $\frac{(n^{2}+nk+n-k)(n-k-1)!}{(n-2k+1)!k!}$

\strut

$\qquad \qquad =$ $\frac{(n+1)(n-k)!}{k(2n-2k+1)(k-1)!}$ $=$ $\frac{n+1}{k}%
\left( 
\begin{array}{c}
(n+1)-k-1 \\ 
k-1
\end{array}
\right) ,$

\strut and hence,

(4.7)

\begin{center}
$
\begin{array}{ll}
g(n,\text{ }k)+g(n-1,k-1) & =\frac{n+1}{k}\left( 
\begin{array}{c}
(n+1)-k-1 \\ 
k-1
\end{array}
\right) \\ 
&  \\ 
& =g(n+1,\text{ }k).
\end{array}
$
\end{center}

Therefore, one can get the recurrence relation (4.4), by (4.6) and (4.7).
\end{proof}

\strut By (4.4), we can realize that the family

\begin{center}
$\{g(n,$ $k)$ $:$ $(n,$ $k)$ $\in $ $\Bbb{N}\times \Bbb{N}_{0}\}$
\end{center}

satisfies the following diagram

(4.8)

\strut

\begin{center}
\strut $
\begin{array}{ccccccccccc}
& & & & & & & & & & 2 \\
& & &  & &  & &  &  &  1 \\
& & &  & &  & &  &  1  & & 2 \\
 &  &  &  &  &  &  & 1 & & 3 & \\
 &  &  &  &  &  & 1 &  & 4 &  & 2 \\
 &  &  &  &  & 1 &  & 5 &  & 5 &  \\
 &  &  &  & 1 &  & 6 &  & 9 &  & 2 \\
 & &  & 1 & & 7 &  & 14 &  & 7 &  \\
 &  & 1 &  & 8 &  & 20 & & 16 & & 2  \\
 & 1 &  & 9 &   & 27 &  & 30 &  & 9 &  \\
1 &  & 10 &  & 35 &  & 50 &  & 25 & & 2 \\
&  &  &  &  &  &  &  &  &  \\ 
&  &  & \vdots &  &  & \vdots &  &  & 
\end{array}
$
\end{center}

\strut \strut \strut where the rows of (4.8) represents $n$ $\in $ $\Bbb{N},$
and the columns of (4.8) represents $k$ $\in $ $\Bbb{N}_{0}$ of the
quantities $g(n,$ $k).$\strut

\strut

\textbf{Observation} The above diagram (4.8) is nothing but the \emph{Lucas
triangle} (e.g., [11]) induced by the \emph{Lucas numbers} (e.g., [7], [8],
[9] and [10]). $\square $

\strut

By using the above family $\left\{ g(n,\text{ }k)\right\} _{(n,\text{ }k)\in 
\Bbb{N}\times \Bbb{N}_{0}}$ of quantities obtained from (4.3), one can
re-write the $n$-th entries $f_{n}$ of (4.1) as follows;

(4.9)

\begin{center}
$f_{n}(z)$ $=$ $\sum_{k=0}^{[\frac{n}{2}]}$ $(-1)^{k}$ $\left( g(n,\text{ }%
k)\right) z^{n-2k},$
\end{center}

for all $n$ $\in $ $\Bbb{N}.$

The new expression (4.9) shows that the coefficients of $f_{n}$ is
determined by the Lucas numbers alternatively, by (4.8).

\begin{corollary}
The coefficients of the functions $f_{n}$ of (4.1) are determined by the
Lucas numbers in the Lucas triangle (4.8) alternatively. $\square $
\end{corollary}

\strut Moreover, we can verify that

$\qquad f_{n+1}(z)$ $=$ $\sum_{k=0}^{[\frac{n+1}{2}]}$ $(-1)^{k}\left( g(n+1,%
\text{ }k)\right) z^{n+1-2k}$

\strut

$\qquad \qquad =$ $\sum_{k=0}^{[\frac{n+1}{2}]}$ $(-1)^{k}\left( g(n,\text{ }%
k)+g(n-1,\text{ }k-1)\right) z^{n+1-2k}$

\strut

$\qquad \qquad =$ $\sum_{k=0}^{[\frac{n+1}{2}]}$ $(-1)^{k}\left( g(n,\text{ }%
k)\right) $ $z^{n+1-2k}$

\strut

$\qquad \qquad \qquad \qquad \qquad +$ $\sum_{k=1}^{[\frac{n+1}{2}%
]}(-1)^{k}\left( g(n-1,\text{ }k-1)\right) z^{n-2k}$

\strut

$\qquad \qquad =$ $z\left( \sum_{k=0}^{[\frac{n}{2}]}(-1)^{k}\left( g(n,%
\text{ }k-1)\right) z^{n-2k}\right) $

\strut

$\qquad \qquad \qquad \qquad \qquad -$ $\sum_{k=0}^{[\frac{n-1}{2}%
]}(-1)^{k}\left( g(n-1,\text{ }k-1)\right) z^{n-1-2k}$

\strut

$\qquad \qquad =$ $zf_{n}(z)-$ $f_{n-1}(z).$

\strut

i.e., from the recurrence relation (4.7), we can re-prove the recurrence
relation (3.2.16).

\begin{corollary}
\strut Let

\begin{center}
$f_{n}(z)$ $=$ $\sum_{k=0}^{[\frac{n}{2}]}$ $(-1)^{k}\left( g(n,\text{ }%
k)\right) z^{n-2k}$
\end{center}

be in the sense of (4.9), where $\{g(n,$ $k)\}_{(n,k)\in \Bbb{N}\times \Bbb{N%
}_{0}}$ are in the sense of (4.3). Then

\begin{center}
$f_{1}(z)$ $=$ $z$ and $f_{2}(z)$ $=$ $z^{2}-2,$
\end{center}

and

\begin{center}
$f_{n+1}(z)$ $=$ $zf_{n}(z)-f_{n-1}(z),$
\end{center}

for all $n$ $\geq $ $2$ in $\Bbb{N}.$ $\square $
\end{corollary}

\strut Again, the above corollary demonstrates the connection between our
Hilbert-space representations of $\Gamma _{2}^{2}$ and Lucas numbers.

\strut \strut

\strut \strut \strut \strut

\textbf{References}

\strut \strut

\strut{\small [1] \ \ I. Cho, Free Probability on Hecke Algebras and Certain Group 
}$C^{*}${\small -Algebras Induced by Hecke Algebras, Opuscula Math., (2015)
To Appear.} 

{\small [2] \ \ I. Cho, Representations and Corresponding Operators Induced
	by Hecke Algebras, DOI: 10.1007/s11785-014-0418-7, (2014) To Appear.}

{\small [3] \ \ R. Speicher, Combinatorial Theory of the Free Product with
	Amalgamation and Operator-Valued Free Probability Theory, Momoir, Ame. Math.
	Soc., 627, (1998)}

{\small [4] \ \ D. Voiculescu, K. J. Dykemma, and A. Nica, Free Random
	Variables, CRM Monograph Ser., vol.1, ISBN-13: 978-0821811405, (1992)}

{\small [5] \ \ I. Cho, and T. Gillespie, Free Probability on Hecke
	Algebras, DOI: 10.1007/s11785-014-0403-1, (2014) To Appear.}

{\small [6] \ \ I. Cho, The Moments of Certain Perturbed Operators of the
	Radial Operator in the Free Group Factors }$L(F_{N})${\small , J. Anal.
	Appl., vol. 5, no.3, (2007) 137 - 165.}

{\small [7] \ \ M. Asci, B. Cekim, and D. Tasci, Generating Matrices
	for Fibonaci, Lucas and Spectral Orthogonal Polynomials with Algorithms,
	Internat. J. Pure Appl. Math., 34, no. 2, (2007) 267 - 278.}

{\small [8] \ \ M. Cetin, M. Sezer, and C. S\"{u}ler, Lucas Polynomial
	Approach for System of High-Order Linear Differential Equations and Residual
	Error Estimation, Math. Problems Eng., Article ID: 625984, (2015)}

{\small [9] \ \ A. Nalli, and P. Haukkanem, On Generalized Fibonacci and
Lucas Polynomials, Chaos, Solitons \& Fractals, 42, (2009) 3179 - 3186.}

{\small [10] J. Wang, Some New Results for the }$(p,$ $q)${\small -Fibonacci
and Lucas Polynomicals, Adv. Diff. Equ., 64, (2014) }

{\small [11] A. T. Benjamin, The Lucas Triangle Recounted, available at }

\begin{center}
{\small http://www.math.hmc.edu/\symbol{126}%
benjamin/papers/LucasTriangle.pdf. }
\end{center}

\strut
\strut

\textbf{Student Author}

{\small Ryan Golden is a senior at St. Ambrose University double majoring in Mathematics and Biology with minors in Computer Science and Chemistry.  He plans to pursue a PhD in Applied Mathematics, specifically Mathematical and Computational Neuroscience.  While his main interests are in the applications of mathematics, he is passionate about theoretical mathematics as well, and is indebted to Dr. Cho for providing the opportunity to work with and learn from him throughout his undergraduate career.}

\strut
\strut

\textbf{Press Summary}
In this paper, the authors studied matricial representations of certain finitely
presented groups generated by $N$-generators of order-2. As an
application, they considered corresponding group algebras under representation. Specifically, they characterized the
inverses of all group elements in terms
of matrices in the group algebra. From the study of this
characterization, the authors realized there are close relations between the trace of
the radial operator of the algebras and the Lucas numbers appearing in
the Lucas triangle.

\end{document}